\documentclass[reqno]{amsart}
\input epsf

\usepackage{amscd}
\usepackage{amssymb}
\usepackage{amstext}
\usepackage{amsmath}
\usepackage[all]{xy}
\usepackage{mathrsfs}

\usepackage{color}

\def\E{\ifmmode{\mathbb E}\else{$\mathbb E$}\fi} 
\def\N{\ifmmode{\mathbb N}\else{$\mathbb N$}\fi} 
\def\R{\ifmmode{\mathbb R}\else{$\mathbb R$}\fi} 
\def\Q{\ifmmode{\mathbb Q}\else{$\mathbb Q$}\fi} 
\def\C{\ifmmode{\mathbb C}\else{$\mathbb C$}\fi} 
\def\H{\ifmmode{\mathbb H}\else{$\mathbb H$}\fi} 
\def\Z{\ifmmode{\mathbb Z}\else{$\mathbb Z$}\fi} 
\def\P{\ifmmode{\mathbb P}\else{$\mathbb P$}\fi} 
\def\T{\ifmmode{\mathbb T}\else{$\mathbb T$}\fi} 
\def\SS{\ifmmode{\mathbb S}\else{$\mathbb S$}\fi} 
\def\DD{\ifmmode{\mathbb D}\else{$\mathbb D$}\fi} 

\newcommand{\e}{\varepsilon}

\newcommand{\del}{\partial}

\newcommand{\ben}{\begin{enumerate}}
\newcommand{\een}{\end{enumerate}}
\newcommand{\be}{\begin{equation}}
\newcommand{\ee}{\end{equation}}
\newcommand{\bea}{\begin{eqnarray}}
\newcommand{\eea}{\end{eqnarray}}
\newcommand{\bc}{\begin{center}}
\newcommand{\ec}{\end{center}}
\newcommand{\beastar}{\begin{eqnarray*}}
\newcommand{\eeastar}{\end{eqnarray*}}

\theoremstyle{theorem}
\newtheorem{thm}{Theorem}[section]
\newtheorem{cor}[thm]{Corollary}
\newtheorem{lem}[thm]{Lemma}
\newtheorem{prop}[thm]{Proposition}

\theoremstyle{definition}
\newtheorem{defn}{Definition}[section]
\newtheorem{rem}[defn]{Remark}

\newtheorem{exm}[defn]{Example}

\numberwithin{equation}{section}

\def\R{{\mathbb R}}

\def\E{{\mathbb E}}
\def\Z{{\mathbb Z}}
\def\C{{\mathbb C}}
\def\R{{\mathbb R}}

\def\N{{\mathbb N}}

\def\SS{{\mathcal S}}
\def\LL{{\mathcal L}}

\def\DD{{\mathcal D}}

\def\AA{{\mathcal A}}

\def\FF{{\mathcal F}}

\def\11{{\mathbb I}}

\def\C{\mathbb{C}}
\def\Z{\mathbb{Z}}

\def\T{\mathbb{T}}
\def\L{\mathbb{L}}

\def\Q{\mathbb{Q}}

\def\E{\ifmmode{\mathbb E}\else{$\mathbb E$}\fi} 
\def\N{\ifmmode{\mathbb N}\else{$\mathbb N$}\fi} 
\def\R{\ifmmode{\mathbb R}\else{$\mathbb R$}\fi} 
\def\Q{\ifmmode{\mathbb Q}\else{$\mathbb Q$}\fi} 
\def\C{\ifmmode{\mathbb C}\else{$\mathbb C$}\fi} 
\def\H{\ifmmode{\mathbb H}\else{$\mathbb H$}\fi} 
\def\Z{\ifmmode{\mathbb Z}\else{$\mathbb Z$}\fi} 
\def\P{\ifmmode{\mathbb P}\else{$\mathbb P$}\fi} 
\def\SS{\ifmmode{\mathbb S}\else{$\mathbb S$}\fi} 
\def\DD{\ifmmode{\mathbb D}\else{$\mathbb D$}\fi} 

\def\R{{\mathbb R}}

\def\E{{\mathbb E}}
\def\Z{{\mathbb Z}}
\def\C{{\mathbb C}}
\def\R{{\mathbb R}}

\def\N{{\mathbb N}}
\def\LL{{\mathcal L}}
\def\FF{{\mathcal F}}

\def\e{\varepsilon}

  \def\P{\Psi}

\def\x{\xi}

\def\CF{{\mathcal F}}

\def\CI{{\mathcal I}}

\def\CL{{\mathcal L}}

\def\CS{{\mathcal S}}

\def\darr#1{\raise1.5ex\hbox{$\leftrightarrow$}
\mkern-16.5mu #1}

\def\roughly#1{\raise.3ex\hbox{$#1$\kern-.75em
\lower1ex\hbox{$\sim$}}}

\def\opname#1{\mathop{\kern0pt{\rm #1}}\nolimits}

\def\dim{\opname{dim}}

\def\rank{\opname{rank}}

\begin{document}
\quad \vskip1.375truein

\def\mq{\mathfrak{q}}
\def\mp{\mathfrak{p}}
\def\mH{\mathfrak{H}}
\def\mh{\mathfrak{h}}
\def\ma{\mathfrak{a}}
\def\ms{\mathfrak{s}}
\def\mm{\mathfrak{m}}
\def\mn{\mathfrak{n}}

\def\Hoch{{\tt Hoch}}
\def\mt{\mathfrak{t}}
\def\ml{\mathfrak{l}}
\def\mT{\mathfrak{T}}
\def\mL{\mathfrak{L}}
\def\mg{\mathfrak{g}}
\def\md{\mathfrak{d}}

\title[Deformations of Coisotropic Submanifolds]
{Deformations of Coisotropic Submanifolds \\
in locally conformal symplectic manifolds}

\author{H\^ong V\^an L\^e
\and Yong-Geun Oh}
\thanks{The  first named author is  partially supported by RVO: 67985840,
the second named author is partially supported
by the NSF grant \#DMS 0904197}

\address{
Department of Mathematics, University of Wisconsin, Madison, WI
53706 \& Department of Mathematics, POSTECH, Pohang, Korea\\
\and\\
Institute of Mathematics of ASCR, Zitna 25, 11567 Praha 1, Czech Republic}
\begin{abstract}
In this paper, we study deformations of coisotropic submanifolds
in a locally conformal symplectic manifold. Firstly, we derive the equation
that governs $C^\infty$ deformations of coisotropic submanifolds and define the cor-
responding $C^\infty$-moduli space of coisotropic submanifolds modulo the Hamil-
tonian isotopies. Secondly, we prove that the formal deformation problem is
governed by an $L_\infty$-structure which is a $\frak b$-deformation of strong homotopy Lie
algebroids introduced in \cite{oh-park} in the symplectic context. Then we study de-
formations of locally conformal symplectic structures and their moduli space,
and the corresponding bulk deformations of coisotropic submanifolds. Finally
we revisit Zambon's obstructed infinitesimal deformation \cite{zambon} in this enlarged
context and prove that it is still obstructed.
\end{abstract}



\maketitle

\tableofcontents

\section{Introduction}
\label{sec:intro}

Symplectic manifolds $(M,\omega)$ have been of much interest in \emph{global} study of
Hamiltonian dynamics, and symplectic topology via analysis of pseudoholomorphic
curves. In this regard closedness of the two-form $\omega$ plays an important
role in relation to the dynamics of Hamiltonian diffeomorphisms and the global
analysis of pseudoholomorphic curves. On the other hand when one takes the
coordinate chart definition of symplectic manifolds and implements the
covariance property of Hamilton's equation, there is no compulsory reason why one should require the
two-form to be closed. Indeed in the point of view of canonical formalism in
Hamiltonian mechanics and construction of the corresponding \emph{bulk physical space}, it is
more natural to require the locally defined canonical symplectic forms
$$
\omega_\alpha = \sum_{i=1}^n dq_i^{\alpha} \wedge dp_i^{\alpha}
$$
to satisfy the cocycle condition
\be\label{eq:cocylce}
\omega_\alpha = \lambda_{\beta\alpha} \omega_\beta, \quad \lambda_{\beta\alpha}\equiv \text{const.}
\ee
with $\lambda_{\gamma\beta}\lambda_{\beta\alpha} = \lambda_{\gamma\alpha}$
as the gluing condition. (See introduction \cite{vaisman:lcs} for a nice
explanation on this point of view) The corresponding bulk constructed in this way
naturally becomes \emph{locally conformal symplectic manifolds}
(abbreviated as l.c.s manifolds) whose definition we first recall.
For the consistency of the definition, we
will mostly assume $\dim M > 2$ in this paper.

\begin{defn} An l.c.s. manifold is a triple $(X, \omega, \frak b)$ where $\frak b$ is
a closed one-form and $\omega$ is a nondegenerate 2-form satisfying the relation
\be\label{eq:relation}
d\omega + \frak b \wedge \omega = 0.
\ee
\end{defn}

We refer to \cite{vaisman:lcs}, \cite{haller-ryb}, \cite{banyaga:lcs}, \cite{Banyaga2007}
for more detailed discussion of general properties of l.c.s. manifolds and non-trivial examples.

Locally by choosing $\frak b = dg$ for a local function $g: U \to \R$
on an open neighborhood $U$, \eqref{eq:relation} is equivalent to
\be\label{eq:local}
d(e^{-g}\omega) = 0
\ee
and so the local geometry of l.c.s manifold is exactly the same as that of
symplectic manifolds. In particular one can define the notion of Lagrangian
submanifolds, isotropic submanifolds, and coisotropic submanifolds in the same
way as in the symplectic case since the definitions require only nondegeneracy of
the two-form $\omega$.

The main question of our interest in this paper is whether the global
geometry of coisotropic submanifolds is any different from that of symplectic case.

We recall that by the results from \cite{zambon}, \cite{oh-park}, deformation theory of
coisotropic submanifolds in symplectic manifolds is generally obstructed. In particular, the set of
coisotropic submanifolds with a given rank does not form a smooth
Frechet manifold \cite{zambon}, and the relevant (formal) deformation theory thereof
is described by an $L_\infty$-structure called \emph{strong homotopy Lie algebroids}
\cite{oh-park}. In the present paper, we show that Oh-Park's deformation theory
naturally extends to that of l.c.s. manifolds, once appropriate normal form theorem
of canonical neighborhoods of coisotropic submanifolds (Theorem \ref{thm:normalform})
and the theory of bulk-deformed strong homotopy Lie algebroids (sections \ref{sec:deform-shla},
\ref{sec:bulk}) are developed.
For this purpose, we need to prove the l.c.s analog of Darboux-type theorem \cite{alan:darboux} and
develop the l.c.s. analog to Moser's trick, for which usage of Novikov-type cohomology
instead of the ordinary de-Rham cohomology is essential. (See \cite{haller-ryb} for relevant exposition
of this cohomology theory.) We derive   two equivalent equations that govern $C^\infty$-deformations  of coisotropic submanifolds (Theorems \ref{thm:coisotropic}, \ref{thm:2})  and  develop   a theory of bulk deformations  of
l.c.s. forms and of coisotropic  submanifolds in this larger context of l.c.s. manifolds.

Some more motivations of the present study are in order. First of all, we would like to
see if the obstructed example of Zambon \cite{zambon} in the symplectic context is
still obstructed in this enlarged deformations of coisotropic submanifolds together with
bulk deformations of l.c.s. structures with replacement of closedness of $\omega$ by
the Novikov-closedness of $\frak b$-twisted differential.
 We then prove that Zambon's
example still remains obstructed even under this enlarged setting of
bulk deformations (Theorem \ref{thm:zambonlcs}).

Another source of motivation comes from the study of $J$-holomorphic curves in this
enlarged bulk of l.c.s. manifolds. Again all the local theory of $J$-holomorphic
curves go through without change. The only difference lies in the global
geometry of $J$-holomorphic curves and it is not completely clear at this moment
whether Novikov-closedness of l.c.s. structure $(X,\omega,\frak b)$ would give
reasonably meaningful implication to the study of moduli problem of $J$-holomorphic
curves in the context of closed strings or open strings attached to
suitably physical $D$-branes. We refer to \cite{kapustin-orlov} for some
physical motivation of coisotropic D-branes and to \cite{cattaneo} for a generalization of study of
deformations of coisotropic submanifolds in the Poisson context.

We would like to thank Yoshihiro Ohnita for inviting us to the Pacific
Rim Geometry Conference in 2011 where the first named author gave a talk on l.c.s. manifolds,
which triggered our collaboration.

\section{Locally conformal pre-symplectic manifolds}
\label{sec:l.c.p-s.}

Suppose $Y \subset (X,\omega_X,\frak b)$ is a coisotropic submanifold.
Then the restriction $(Y,\omega,b)$ satisfies the same equation
\be\label{eq:dbomega=0}
d^b\omega: = d\omega + b \wedge \omega = 0
\ee
except that $\omega$ is no longer nondegenerate but has
constant rank.

This gives rise to the notion of \emph{locally conformal pre-symplectic manifolds},
abbreviated as l.c.p-s. manifold.

\begin{defn} A triple $(Y,\omega,b)$ is called an l.c.p-s. manifold if
$b$ is a closed one-form and $\omega$ is a two-form with constant rank that
satisfy
\be\label{eq:dbomega=0}
d^b\omega: = d\omega + b \wedge \omega = 0.
\ee
\end{defn}

\begin{rem}\label{rem:onb} If the rank of $\omega$ is at least 4,
then the  wedge product with $\omega$   defines a linear injective map  from  $\Omega^1(Y)$  to $\Omega^3 (Y)$.
Hence $b$ is defined  uniquely by the equation (\ref{eq:dbomega=0}).
If  the rank of $\omega$ is 2, then the  restriction of $b$ to the  null space $TY ^\omega$  of $\omega$ is  defined by (\ref{eq:dbomega=0}). The kernel of the wedge product $ \Omega^1(Y) \to \Omega ^3(Y),
\, \gamma \mapsto \omega\wedge \gamma,$ is the two-dimensional  annihilator
$T\FF ^\circ$  of $TY^\omega$.  In particular,  if  rank $\omega$ is 2
and $(Y, \omega, b)$ is an l.c.p-s. manifold, then $(Y, \omega, b + b')$ is
also an l.c.p-s. manifold for any $b'\in T\FF ^\circ$ such that $db'= 0$.
\end{rem}

From now on, we consider a general l.c.p-s. manifold $(Y,\omega,b)$.

We next introduce morphisms between l.c.p-s. manifolds
and automorphisms of  $(Y,\omega,b)$ generalizing those of l.c.s.
manifolds (see \cite{haller-ryb} for the corresponding definitions for the
l.c.s. case.)

\begin{defn}\label{morphism} Let $(Y,\omega,b)$ and
$(Y^\prime, \omega^\prime, b ^\prime)$ be two l.c.p-s. manifolds. A
diffeomorphism $\phi: Y\to Y'$ is called {\it l.c.p-s.} if there exists
$a \in C^\infty(Y,\R\setminus \{0\})$ such that
$$
\phi^*\omega^\prime = (1/a) \omega, \quad \phi^*b' = b + d(\ln|a|).
$$
\end{defn}

By setting $a =  e^{tu}$, it is easy to check that the following is the infinitesimal
version of Definition \ref{morphism}.

\begin{defn}\label{vectorfield}
Let $(Y,\omega,b)$ be a l.c.p-s. manifold. A vector field $\xi$
on $Y$ is called  {\it l.c.p-s.} if there exists a function $u \in C^\infty(Y)$ such
that
$$
\LL_\xi \omega = - u \omega, \quad \LL_\xi b = du
$$
\end{defn}

We denote by $Diff(Y,\omega,b)$ the set of l.c.p-s. diffeomorphisms.

\begin{defn} We call any such function $u \in C^\infty(Y)$ that appears in Definition
\ref{vectorfield} is called an l.c.p-s. function. We denote by $C^\infty(Y;\omega,b)$
the set of l.c.p-s. functions.
\end{defn}

It is easy to see that $C^\infty(Y;\omega,b)$ is a vector subspace of $C^\infty(X)$.
%
%

We say an l.c.p-s. diffeomorphism (resp. vector field) an
l.c.s. diffeomorphism (resp. vector field), if $(Y,\omega,b)$ is an l.c.s. manifold.

\section{Canonical neighborhoods of coisotropic submanifolds}
\label{sec:neighborhoods}

In this section, we develop the l.c.s. analog to Gotay's
coisotropic neighborhood theorem [Go] in the symplectic case.

As in the symplectic case, we denote
$$
E= (TY)^\omega
$$
the {\it characteristic distribution} on $Y$. The following lemma
is one of the important ingredients that enables us to develop deformation theory of coisotropic
submanifolds in l.c.s. manifolds in a way similar
to the symplectic case as done in \cite{oh-park}.

\begin{lem} The distribution $E$ on $Y$ is integrable.
\end{lem}
\begin{proof} This is an immediate consequence \eqref{eq:dbomega=0}
which shows that the ideal generated by $\omega$ is a differential
ideal.
\end{proof}
We call the corresponding foliation the {\it null foliation} on
$Y$ and denote it by $\CF$.

We now consider the dual bundle $\pi: E^* \to Y$ of $E$.
The bundle
$TE^*|_Y$ where $Y \subset E^*$ is the zero section of $E^*$
carries the canonical decomposition
$$
TE^*|_Y = TY \oplus E^*.
$$
In the standard notation in the foliation theory, $E$ and $E^*$ are denoted by
$T\FF$ and $T^*\FF$ and called the tangent bundle (respectively
cotangent bundle) of the foliation $\FF$.

\begin{rem}\label{rem:NY=E*} When $Y$ is a coisotropic submanifold of an l.c.s. manifold
$(X,\omega_X,\frak b)$ and $(Y,\omega,b)$ is the associated l.c.p-s. structure, then
it is easy to check that the canonical isomorphism
$$
\widetilde \omega_0: TX \to T^*X
$$
maps $E = TY^\omega$ to the conormal $N^*Y \subset T^*X$
and so its adjoint $(\widetilde \omega_0)^\dagger: TX \to T^*X$
induces an isomorphism between $NY = TX/TY$ and $E^*$ where
$E = (TY)^\omega$.
\end{rem}

We choose a splitting
\begin{equation}\label{eq:splitting}
TY = G \oplus E, \quad E = (TY)^\omega,
\end{equation}
and denote by
$$
p_G: TY \to E
$$
the projection to $E$ along $G$ in the splitting
(\ref{eq:splitting}).
Using this splitting, we can write a conformal symplectic form on a
neighborhood of the zero section $Y \hookrightarrow E^*$ in the
following way similarly as in the symplectic case.  We have the bundle map
\begin{equation}\label{eq:bundlemap}
TE^* \stackrel{T\pi}\longrightarrow TY \stackrel{p_G}
\longrightarrow E.
\end{equation}

Let $\widehat\alpha\in E^*$ and $\xi \in T_{\widehat \alpha} E^*$.
We define the one-form $\theta_G$ on $E^*$ by its value
\begin{equation}\label{eq:thetag}
\theta_{G,\widehat\alpha}(\xi): = \widehat\alpha(p_G\circ T\pi(\xi))
\end{equation}
at each $\widehat \alpha \in E^*$.
Then the two form
\begin{equation}\label{eq:omega*}
\omega_G:= \pi^*\omega - d\theta_G - \pi^*b \wedge \theta_G
\end{equation}
is non-degenerate in a neighborhood $U \subset E^* $ of the zero
section (See the coordinate expression (\ref{eq:dthetaG}) of $d
\theta_G$ and $\omega_G$). Later, we  use $\omega _G$ and $\omega _U$
 interchangeably depending on context.

Then a straightforward computation proves

\begin{prop}
Then the pair $(U, \omega_U, \frak b_U)$ with $\frak b_U := \pi^*b|_U$
defines an l.c.s. structure.
\end{prop}

\begin{rem}\label{rem:lcs} If $Y$ is Lagrangian, then $E = TY$,
$ E^*_{| Y} = T^* Y$, and  hence  $p_G = Id$. For this special case,    Proposition \label{pro:lcs}
 is known as Example 3.1 in \cite{haller-ryb}, where $\theta_G$ is the Liouville 1-form on $T^* Y$.
\end{rem}

In the next section, we will prove that this provides a general normal
form of the l.c.s. neighborhood of the triple $(Y,\omega,b)$ which depends only on $(Y,\omega,b)$
and the splitting (\ref{eq:splitting}), and that
this normal form is unique up to diffeomorphism.  We call
the pair $(U,\omega_U,\frak b_U)$ a {\it (canonical) l.c.s.
 thickening}
of the l.c.p-s. manifold $(Y,\omega,b)$.

\section{Normal form theorem of coisotropic submanifolds in l.c.s. manifold}
\label{sec:coiso-normal}

Let $Y$ be a compact coisotropic submanifold
in a l.c.s. manifold $(X, \omega_X, \frak b)$. Denote by $(Y,\omega,b)$
the induced l.c.p-s. structure given by
$$
\omega = i^*\omega_X, \, b = i^*\frak b,
$$
where $i : Y \to X$ is the canonical embedding. We denote by $(Ti)^*$
the associated  bundle map $T^*X  \to T^* Y$.
Consider the bundle $\pi: E^* \to Y$  where
$E = (TY)^\omega= T\CF$ as in the previous section.

By Remark \ref{rem:NY=E*}, the adjoint isomorphism
$$
(\widetilde \omega_0^\dagger): TX \to T^*X
$$
induces an isomorphism
\begin{equation}
\widetilde \omega_X: NY = TX/TY \to E^*.\label{eq:wto}
\end{equation}

More precisely, we have the following lemma

\begin{lem}\label{lem:ny} The 
nondegenerate two-form $\omega_X$ induces a canonical
bundle isomorphism
$
\widetilde \omega_X: TX|_Y/TY \to E^*
$
given by \eqref{eq:[v]}.
\end{lem}

\begin{proof} We define the bundle map
\be\label{eq:vtoomegaXv}
\widetilde \omega_{XY}: TX|_Y \to T^*Y, \,  v \mapsto(Ti)^* (v \rfloor \omega_X).\nonumber 
\ee

Denote by $j^*: T^* Y \to E ^*$ the adjoint of  $j: E \to TY$. Then $E = \ker \omega$  implies that
$$
TY \subset \ker (j ^*  \circ \widetilde \omega_{XY}).
$$
Hence $j ^* \circ \widetilde \omega_{XY}$ descends to a  bundle map $\widetilde \omega_X: NY = TX|_Y/TY \to E^*$
by setting
\be\label{eq:[v]}
[v] \in NY \mapsto j ^*\circ \widetilde \omega_{XY}(v).
\ee
The nondegeneracy of $\omega_X$ 
implies that
$\widetilde \omega_X $ induces a canonical bundle isomorphism.
\end{proof}

Using this isomorphism, we identify the pair $(NY,Y)$ with $(E^*,o_{E^*})$.

Now let $g$ be a   Riemannian metric  on $X$. Then $g$ gives rise to a splitting
$$
TX|_Y \cong TY \oplus NY.
$$
We also have a canonical isomorphism
\be\label{eq:identify}
TE^*|_Y \cong TY \oplus E^*.
\ee
Combining (\ref{eq:identify}) with Lemma \ref{lem:ny} we get
$$
TE^*|_{o_{E^*}} \cong To_E^* \oplus E^* \cong TY \oplus NY \cong \quad TX|_Y.
$$
Through this identification, we regard a neighborhood $U_1 \subset X$ of $Y$ as a neighborhood of the zero section
$o_{E^*} = Y \subset E^*$.

For any open set $U \subset X$ we denote the restriction of
$\omega_X$ (resp. $ \frak b$) to $U$ also by $\omega _X$ (resp. by $\frak b$).
In this section, we prove the following normal form theorem.

\begin{thm}[Normal form]\label{thm:normalform}  Assume that $Y$ is compact.
There exist  an open neighborhood  $U\subset X$ of  $Y$,
a neighborhood  $V \subset E^*$ of  the zero  section $Y$,   and  a l.c.s. diffeomorphism
$$
\phi: (U,\omega_X,\frak b) \to (V, (\omega_G)|_V, \pi^*b)
$$
such that $\phi|_Y = Id$ and  $d\phi|_{NY} = \widetilde \omega_X$ under the identification \eqref{eq:identify}.
More specifically, $\phi$ satisfies
$$
\phi^*(\pi^*\omega - d\theta_G - \pi^*b \wedge \theta_G) = e^{-f} \omega_X
$$
for some $f \in C^\infty(U)$.
\end{thm}
\begin{proof}  Assume that $U_1$ be  a neighborhood of $Y$ in $X$  which can be
identified with a neighborhood $W_1$ of $Y$ in $NY$  via the exponential map  $Exp_Y: NY \to U_1$.
Set
\begin{equation}
\omega_1 : = Exp_Y^*(\omega_X), \quad \frak b_1 : = Exp_Y^*(\frak b).\label{eq:omega1}
\end{equation}
Then $(W_1,\omega_1, \frak b_1)$ is a l.c.s. manifold. Since the restriction of
$dExp_Y$ to $Y$ is equal to identity,
$Y$ is also a coisotropic submanifold in $(W_1,\omega_1, \frak b_1)$.
Let $ i_X: Y \hookrightarrow X$ be  the inclusion. Set 
\begin{eqnarray}
V &: =& \widetilde \omega_X(W_1),  \quad  b  : =   i_X^*\frak b\in \Omega ^1 (Y), \nonumber\\
\widetilde \omega _1 &: = & (\widetilde \omega_X ^{-1} ) ^* (\omega_1)\in \Omega^2 (V),
\quad \widetilde {\frak b}_1: = (\widetilde \omega_X ^{-1} ) ^* (\frak b_1)\in \Omega^1 (V).
\end{eqnarray}
Denote by $i_{E^*}: Y \hookrightarrow E^*$ the inclusion as the zero section
and  by $H_b ^* (Y)$  the cohomology group $\ker d ^b / \mathrm{ im }\, d ^b$.

\begin{lem}\label{lem:sheaf} The embedding $i_{E^*} : Y \to E^*$ induces an isomorphism between
$H_b^*(Y)$ and $H_{\pi^*b}^*(E^*)$. 
In particular, there exists a one-form $\eta \in \Omega^1(E^*)$
such that $\widetilde\omega_1 - \pi^*(\omega_1|_Y) = d^{\pi^*b}(\eta)$.
\end{lem}
\begin{proof} Denote by $\CS$ the   following locally constant sheaf  on $Y$
$$U \mapsto \CS (U): = \{f \in C^\infty(U, \R)|\, d ^{b|_U} f = 0\}.$$
It is known that $H_{b} (Y) = H(Y,\CS)$, see e.g. \cite[Remark 1.10]{haller-ryb}.
The first assertion  of Lemma \ref{lem:sheaf} follows from the homotopy invariant
property of  cohomology with values in locally constant sheaf.
The second assertion of Lemma \ref{lem:sheaf} is a consequence of the first assertion.
\end{proof}

Since $H^1(E^*, \R) = H^1(Y, \R)$ there exists a function $ f \in C^\infty (E^*)$ such that
$\eta = \pi^*(i^*\eta) + df$. 
 Then $e^{f} \widetilde \omega$ is an l.c.s.  form on $E^*$
with the Lee form $\widetilde {\frak b}_1 - df  = \pi^*(\frak b|_Y)$
\cite{Le1943}.

Now we apply Moser's deformation to the normal form. Set
$$
\widetilde \omega_0:=\omega_{G}|_V.
$$
By (\ref{eq:omega*}) and (\ref{eq:omega1}) we have
\begin{equation}
\widetilde\omega_0(y) = \widetilde\omega_1(y)\text {  for all } y \in Y.\label{eq:ini}
\end{equation}
Since
$H^*_{\pi^*b}(V) \cong H^*_{\widetilde{\frak b}_1}(V)$, there exists a one form $\sigma$
on $V$ such that
$$
\widetilde\omega_1 - \widetilde\omega_0 = d^{\pi^*b}\sigma.
$$
Set
$$
\widetilde\omega_t: = \widetilde\omega_0 + td^{\pi^*b} \sigma = \pi^*(\omega_1|_Y) - d^{\pi^*b} (\theta_G - t \sigma)
$$
By making $V$ smaller if necessary, taking into account (\ref{eq:ini})  and  the compactness of $Y$,
we assume that $\widetilde\omega_t$ are nondegenerate
for all $t \in [0,1]$.
To prove Theorem \ref{thm:normalform}, it suffices to solve the equation
\be\label{eq:tosolve}
\psi_t^* (\widetilde\omega_t) = e ^{f_t(x)} \widetilde\omega_0
\ee
for a family of diffeomorphism $\psi_t $ of $V$ and a function $f_t$ with $f_1 = f$.
Let $\xi_t$ be the generating  vector field of $\psi_t$ i.e.
$$
{d\over dt} \psi_t = \xi_t (\psi _t), \, \psi_0 = Id.
$$
Differentiating \eqref{eq:tosolve}, we obtain
$$
\psi_t^*\left({d\over dt} \widetilde\omega_t +  \CL_{\xi_t}\widetilde \omega_t\right)
= \frac{\del f}{\del t}\widetilde \omega_0
$$
which is equivalent to
$$
{d\over dt} \widetilde\omega_t +  \CL_{\xi_t}\widetilde \omega_t = \frac{\del f}{\del t} \circ \psi_t^{-1}.
$$
But by definition of $\omega_t$ and Cartan's formula, this becomes
$$
d^{\pi^*b} \sigma + \xi_t \rfloor d\widetilde\omega_t + d(\xi_t \rfloor\widetilde \omega_t) = g_t,
\quad g_t(x) = \frac{\del f}{\del t}(\psi_t^{-1}(x)).
$$
which in turn becomes
$$
g_t  =  d^{\pi^*b} \sigma - \xi_t \rfloor(\pi^*b \wedge\widetilde \omega_t)
+ d^{\pi^*b} (\xi_t\rfloor\widetilde \omega_t)
- \pi^*b \wedge (\xi_t \rfloor \widetilde\omega_t).
$$
In other words, we obtain
$$
(g_t - \pi^*b(\xi_t))\widetilde \omega_t = d^{\pi^*b} (\sigma + \xi_t \rfloor\widetilde \omega_t ).
$$
Using non-degeneracy of $\widetilde\omega_t$, we first solve
$$
\sigma + \xi_t \rfloor \widetilde\omega_t = 0
$$
for $\xi_t$ on $V$ and then define $g_t$ by
$$
g_t = \pi^*b (\xi_t)
$$
for all $(t,x) \in [0,1] \times V$ again shrinking $V$, if necessary.
We denote by $\psi_t$ the flow of $\xi_t$ which then determines $f_t$ by
$f_t = g_t \circ \psi_t$.

This proves Theorem \ref{thm:normalform}.
\end{proof}

\section{Geometry of the null foliation of l.c.p-s. manifold}
\label{sec:nullfoliation}

Let $(Y,\omega,b)$ be
an l.c.p-s. manifold  of dimension $n+k$ and denote by $\FF$ the associated null
foliation.   Set $2n: = \dim  X$, $ n - k: = \dim  \FF$, $l : = 2k$. 
We now
formulate a uniqueness statement
in the symplectic thickening of $(Y,\omega)$ (Proposition \ref{diffeo}), extending an analogous  result in \cite{oh-park}.
We also prove the existence of a transverse   l.c.s. form (Proposition \ref{tr-sy-form}), which is  important for  later sections.

Recalling that the l.c.s. form $\omega_{G}$ of \eqref{eq:omega*}
depends on the choice of the splitting $\Pi$, in this section we redenote by $\omega_\Pi$
the l.c.s. form $\omega_{G}$ associated to the splitting $\Pi$.

\begin{prop}\label{diffeo}(cf. \cite[Proposition 5.1]{oh-park}) For given two splittings $\Pi, \, \Pi'$, there exist
neighborhoods $U, \, U'$ of the zero section $Y \subset E^*$ and a
diffeomorphism $\phi: U \to U'$ and a function $f: U \to \R$ such that
\begin{enumerate}
\item $\phi^*\omega_{\Pi'} = e^f \omega_{\Pi}$,
\item $\phi|_Y \equiv
id$, and $T\phi|_{T_YE^*} \equiv id$ where $T_YE^*$ is the
restriction of $TE^*$ to $Y$.
\end{enumerate}
\end{prop}
\begin{proof}
Since $\AA_E(TY)$ is contractible, 
we can choose a smooth family
$$
\{\Pi_t\}_{0 \leq t\leq 1}, \quad \Pi_0=\Pi, \, \Pi_1 = \Pi'.
$$
Denoting $\omega_t :=\omega_{\Pi_t}$, applying the isomorphism $H^1_b(Y) \cong H^1_{\pi^*b}(E^*)$, we have
$$
\omega_t - \omega_0 = d^{\pi^*b} \sigma_t.
$$
From the definition, we have
$$
\sigma_t|_{T_YE^*} \equiv 0.
$$
for all $0 \leq t \leq 1$. With these, we imitate the proof of Theorem \ref{thm:normalform}
to finish off the proof.
\end{proof}

For the study of the deformation problem of l.c.p-s.
structures it is crucial to understand the transverse geometry of
the null foliation. First we note that the l.c.p-s. form
$\omega$ carries a natural {\it transverse l.c.p-s. form}. This
defines the l.c.s. analog to the transverse symplectic form to the
null foliation of pre-symplectic manifold. (See \cite{gotay}, \cite{oh-park},
for example).

\begin{prop}\label{tr-sy-form}(cf. \cite[Proposition 5.2]{oh-park})   Let $\FF$ be the null foliation of
the l.c.p-s. manifold $(Y,\omega,b)$. Then it defines a transverse l.c.s. form
on $\FF$ in the following sense:
\enumerate
\item $\ker (\omega_x) = T_x\FF$ for any $x \in Y$, and
\item $\LL_\xi\omega = - b(\xi) \omega$ for any vector field $\xi$ on $Y$ tangent to $\FF$.
\endenumerate
\end{prop}
\begin{proof} The first statement is trivial by definition of the
null foliation and the second is an immediate consequence of the
Cartan identity
$$
\LL_\xi\omega = d(\xi \rfloor \omega) + \xi \rfloor d\omega.
$$
The first term vanishes since $X$ is tangent to the null foliation $\FF$.
On the other hand, the second term becomes
$$
\xi \rfloor d\omega =- \xi \rfloor (b \wedge \omega) = -b(\xi) \omega + b \wedge (\xi \rfloor \omega)
=- b(\xi) \omega
$$
which finishes the proof.
\end{proof}

One immediate consequence of the presence of the transverse l.c.p-s.
form above is that any {\it transverse section} $T$ of the
foliation $\FF$ carries a natural l.c.s. form: in any
foliation coordinates, it follows from $E = \ker \omega =
\operatorname{span}\{\frac{\del}{\del q^\alpha}\}_{1 \leq \alpha
\leq n-k}$ that we have
\begin{equation}\label{eq:piomegaY}
\pi^*\omega = \frac{1}{2}\sum_{2k \ge i> j\ge 1} \omega_{ij} dy^i \wedge dy^j,
\end{equation}
where $\omega_{ij} = \omega(\frac{\del}{\del y^i},\frac{\del}{\del
y^j})$ is skew-symmetric and invertible.

The condition (2) above implies
\be\label{eq:omegaijbeta}
\omega_{ij;\beta} = b_\beta \omega_{ij}
\ee
where $b = \sum_j b_j dy^j + \sum_\beta b_\beta dq^\beta$ in the foliation coordinates
$(y^1, \cdots, y^{2k}, q^1,\cdots q^{n-k})$.
Note that this expression is {\it independent} of
the choice of splitting as long as $y^1, \cdots, y^{2k}$ are those
coordinates that characterize the leaves of $E$ by
\begin{equation}\label{eq:leaves}
y^1 = c^1, \cdots, y^{\ell} = c^{\ell}, \quad \mbox{ $c^i$'s constant}.
\end{equation}
By the closedness of the one-form $b$, \eqref{eq:omegaijbeta}
gives rise to the following proposition.

\begin{prop}\label{holonomy}
Let $L$ be a leaf of the null foliation $\FF$ on $(Y,\omega)$,
$\lambda$ a path in $L$, and let $T$ and $S$  be transverse
sections of $\FF$ with $\lambda(0) \in T$ and $\lambda(1) \in S$.
Then the holonomy map
$$
hol^{S,T}(\lambda): (T,\lambda(0)) \to (S,\lambda(1))
$$
defines the germ of a l.c.s. diffeomorphism. In particular, each transversal $T$
to the null foliation carries a natural holonomy-invariant l.c.s. structure.
\end{prop}

\section{Master equation and equivalence relations; classical part}
\label{sec:master}
Let us recall  the proof of the fact the a graph of a 1-form $s\in \Omega  ^1(L)$
is Lagrangian with respect to the canonical symplectic form
on $TL$ if and only if $ds = 0$. This fact is a direct consequence of the following formula
$$s^* (\theta) = s,$$
which is obtained  by
$$
\langle s^*(\theta), \delta q\rangle =  \langle \theta, s_* (\delta  q) \rangle
= s (\pi _* s_* \delta q) = s(\delta q)
$$
where $\delta q$ stands for the infinitesimal variation of $q$.
Similarly we will derive the second equation  for the graph $\Gamma_s$ of a section $s: Y
\to E^*\cong NY$  to be coisotropic with respect to $\omega_G$ (Theorem \ref{thm:coisotropic}). We  also
call the corresponding equation  the {\it classical part} of the master equation (cf. Theorem \ref{thm:2}).
We will study the full (local) moduli (with respect  to different equivalence relations) problem of
coisotropic submanifolds by analyzing the condition that the graph
of a section $s: Y \to U$ in the symplectic thickening $U$ is to
be coisotropic with respect to $\omega_G$ (Lemmas \ref{ham=coh}, \ref{lem:equi}).

\subsection{Description of coisotropic Granssmannian}
\label{subsec:grassmannian}

In this section, we recall some basic algebraic facts on
the coisotropic subspace $C$ (with real dimensions $n+k$ where $0\leq k\leq n$)
in $\C^n$ from \cite{oh-park}.  We denote by
$C^{\omega}$ the $\omega$-orthogonal complement of $C$ in
$\R^{2n}$ and by $\Gamma_k$ the set of coisotropic subspaces of
$(\R^{2n}, \omega)$. In other words,
\begin{equation}
\label{gammak} \Gamma_k = \Gamma_k(\R^{2n}, \omega) = :\{ C \in
Gr_{n+k}(\R^{2n}) \mid C^\omega \subset C\}.
\end{equation}
{}From the definition, we have the
canonical flag,
$$
0 \subset C^\omega \subset C \subset \R^{2n}
$$
for any coisotropic subspace.
We call $(C,C^\omega)$ a {\it coisotropic pair}. Combining this
with the standard complex structure on $\R^{2n} \cong \C^n$, we
have the splitting
\begin{equation}\label{eq:csplit}
C = H_C \oplus C^\omega
\end{equation}
where $H_C$ is the complex subspace of $C$.

Next we give a parametrization of all the coisotropic subspaces
near given $C \in \Gamma_k$. Up to the unitary change of
coordinates we may assume that $C$ is the canonical model
$$
C = \C^k \oplus \R^{n-k}.
$$
We denote the (Euclidean) orthogonal complement of $C$ by $C^\perp
= i \R^{n-k}$ which is canonically isomorphic to $(C^\omega)^*$
via the isomorphism $\widetilde\omega: \C^n \to (\C^n)^*$.
Then any nearby subspace of dimension $\dim C$ that
is transverse to $C^\perp$ can be written as the graph of the
linear map
$$
A: C \to C^\perp \cong (C^\omega)^*
$$
i.e., has the form
\begin{equation}\label{eq:ca}
C_A:=\{ (x, Ax) \in C \oplus C^\perp = \R^{2n} \mid x \in C \}.
\end{equation}

Denote $A = A_H \oplus A_I$ where
\begin{align}
A_H & : H = \C^k \to C^\perp \cong (C^\omega)^*,\nonumber\\
A_I & : C^\omega = \R^{n-k}  \to C^\perp\cong (C^\omega)^*.\nonumber
\end{align}
Note that the symplectic form $\omega$ induce the canonical
isomorphism
\begin{eqnarray*}
\widetilde\omega^H & : & \C^k \to (\C^k)^*, \\
\widetilde\omega^I & : & \R^{n-k}=C^\omega
\to (C^\omega)^*\cong C^\perp = i
\R^{n-k} .\end{eqnarray*}
With this identification, the symplectic form $\omega$ has the
form
\begin{equation}\label{eq:omega0}
\omega = \pi^*\omega_{0,k} + \sum_{i=1}^{n-k} dx_i\wedge dy^i,
\end{equation}
where  $\omega_{0,k}$ is the standard symplectic form  in $\C ^ k$,
 $\pi: \C^n \to \C^k$  the projection, and $(x_1,\cdots,
x_{n-k})$ the standard coordinates of $\R^{n-k}$ and $(y^1,
\cdots, y^{n-k})$ its dual coordinates of $(\R^{n-k})^*$. We also
denote by $\pi_H: (\C^k)^* \to \C^k$ the inverse of the above
mentioned canonical isomorphism $\widetilde \omega^H$. 

The following  statements are fundamental in  our work.

\begin{prop} \label{prop:coisotropic}
Let $C_A$ be as in \eqref{eq:ca}.
\begin{enumerate}
\item The subspace $C_A$ is coisotropic if and only
if $A_H$ and $A_I$ satisfies
\begin{equation}\label{eq:grapca}
A_I - (A_I)^* + A_H \pi_H (A_H)^* = 0.
\end{equation}
\item The subspace $C_A$   is coisotropic, if and only if
  $\omega ^{k +1}|_{C_A} = 0$.
\end{enumerate}
\end{prop}
\begin{proof} The first assertion of Proposition \ref{prop:coisotropic} is  Proposition 2.2  in \cite{oh-park}.
The second assertion of Proposition \ref{prop:coisotropic}
follows from the observation that  $C_A$ is coisotropic, if and only  the
restriction $\omega|_{C_A}$ is maximally degenerate, i.e. $\rank (\omega|_{C_A} )= \rank \pi^*\omega_{0,k}$
$ = k$.
\end{proof}

\subsection{The equation for coisotropic sections}
\label{subsec:master}

Note that the projection $p_G: TY \to E$  induces  a bundle map $p_G ^* : E^* \to  T^*Y$ by
$$
\langle p_G^*( s ) , \delta q \rangle :  = \langle s , p_G (\delta q)\rangle
$$
for any  $s \in E^*$ and $\delta q \in T_{\pi (s)}Y$.

As before, assume that $m = \dim Y = \dim E^* - (n-k) = n +k$.

\begin{thm} \label{thm:coisotropic} The graph $\Gamma _s$ of a
section $s: Y \to (U\subset E^*, \omega_G, \pi^* b)$  is coisotropic if and only if the 2-form
$ \omega _G(s) : =\omega|_{ Y} - d ^ b (p_G^*( s))\in \Omega^2 (Y)$
is maximally degenerate, i.e.  $(\omega_ G(s) )^{k+1} =0$.
\end{thm}
\begin{proof} By Proposition \ref{prop:coisotropic}, $\Gamma_s$ is coisotropic
if and only if the restriction of $\omega_G$ to $\Gamma_s$ is maximally degenerate,  or equivalently
$(\omega_G) ^{k+1} _{|\Gamma_s} = 0$.  Since $\omega_G \circ s_* = s^*  (\omega_G)$  we get
\begin{equation}
(\omega_G)^{k+1} |_{\Gamma_s} =  0 \Longleftrightarrow (s^*(\omega_G)| _Y)^{k +1}  = 0.\label{eq:tildeomegag}
\end{equation}
By (\ref{eq:omega*}) we have
\begin{equation}
 s^* (\omega_G) | _Y = s^* ( \pi ^*(\omega|_Y) - d ^{\pi^* b}\theta_G)
 =  \omega|_Y -  s^* (d^{ \pi ^*b} \theta_G) = \omega|_Y -d ^b (s ^* \theta_G). \label{eq:somega}
 \end{equation}
Using (\ref{eq:thetag}) we obtain  for  any  $ y \in Y$ and any $\partial  q \in T_yY$
\begin{equation}
\langle s^*(\theta_G), \delta q\rangle_y =  \langle \theta_G,
s_* (\partial  q) \rangle = s (p_G \circ \pi_*\circ  s_* \delta q) = s(p_G(\delta q)).\label{eq:stheta}
\end{equation}
It follows  from (\ref{eq:somega}) and  (\ref{eq:stheta})
\begin{equation}
s^* (\omega_G) | _Y = (\omega|_Y) - d^b ( p_G^*(s)).\label{eq:somegag}
\end{equation}
Theorem \ref{thm:coisotropic}  follows immediately from  (\ref{eq:somegag}).
\end{proof}

\subsection{(Pre-)Hamiltonian equivalence and  infinitesimal equivalence}
\label{subsec:Ham}

In this section, we
clarify the relation between the {\it intrinsic} pre-Hamiltonian equivalence (resp. intrinsic l.c.p-s. equivalence)
between the l.c.p-s. structures $(\omega, b)$ and the {\it extrinsic}
Hamiltonian equivalence (resp. extrinsic l.c.s. equivalence) between coisotropic embeddings in $(U,\omega_G, \pi^* b)$.  The intrinsic pre-Hamiltonian
equivalence is provided by the pre-Hamiltonian diffeomorphisms  (Definition \ref{def:ham}) on
the l.c.p-s. manifold $(Y,\omega, b)$ and the extrinsic ones by
Hamiltonian deformations of its coisotropic embedding into
$(U,\omega_U, \pi^*b)$ (Definition \ref{def:lcse}). 
 The  intrinsic l.c.p-s. equivalence is provided by l.c.p-s. diffeomorphisms
and the extrinsic ones by l.c.s. deformations of its coisotropic embedding into
$(U,\omega_U, \pi^*b)$.  The infinitesimal ((pre-)Hamiltonian) equivalence is  a natural  infinitesimal version of the intrinsic/extrinsic  ((pre-)Hamiltonian)  equivalence.

First we  shall prove

\begin{lem}\label{lcpsvectorfield} A vector field $\xi$ on an l.c.p-s. manifold
$(Y, \omega, b)$ is  l.c.p-s. if and only if
\begin{equation}
d^b (\xi \rfloor \omega) = c\, \omega \, \text{ for  some } c\in \R. \label{eq:lcpsv}
\end{equation}
\end{lem}

\begin{proof} Assume that $\xi$ is l.c.p-s.  vectorfield on $Y$.
To prove (\ref{eq:lcpsv}) we note that  a l.c.p-s. vector field $\xi$
on a l.c.p-s. manifold $(Y, \omega, b)$ satisfies the following equation
for  some smooth function $u \in C^\infty (Y)$ (see Definition \ref{vectorfield})
\begin{eqnarray}
&{}& \LL _\xi \omega = - u \omega \nonumber\\
&\Longleftrightarrow & \xi \rfloor d\omega  + d( \xi \rfloor \omega) = - u \omega \nonumber\\
&\Longleftrightarrow & - b(\xi)\, \omega + d ^ b (\xi \rfloor \omega) = - u \omega\nonumber\\
& \Longleftrightarrow & d ^b (\xi \rfloor \omega) = (-u + b(\xi)) \, \omega. \label{eq:vectorlcs}
\end{eqnarray}

By Definition (\ref{vectorfield}) $\LL_\xi b = du$,  or equivalently
$$d(b(\xi) - u) = 0.$$
Comparing this with (\ref{eq:vectorlcs}) we  obtain (\ref{eq:lcpsv}) immediately.

Now assume that    a vector field $\xi$  on $Y$  satisfies (\ref{eq:lcpsv}).  Set
$$u: = b(\xi) - c.$$
Then (\ref{eq:vectorlcs}) holds. The  above computations yield
$$\LL_\xi \omega =  - u \omega.$$
Since $\LL_\xi b = d(b(\xi)) = du$, we conclude  that $\xi$ is  l.c.p-s.
This  completes the proof of  Lemma \ref{lcpsvectorfield}.
\end{proof}

We  define a $b$-deformed  Lie derivative as follows
\begin{equation}
\LL_\xi^b (\phi): = d ^b\circ i_\xi + i _\xi \circ d ^b. \label{eq:liederi}
\end{equation}

The following statements are direct consequences of Lemma \ref{lcpsvectorfield},
hence we omit  their  proof.

\begin{cor}\label{cor:lcps}
Let $(Y,\omega,b)$ be an l.c.p-s. manifold.
\begin{enumerate}
\item Assume that $[\omega] \not = 0 \in H^2 _b (Y, \R)$. Then any l.c.p-s.
vector field $\xi$ on $(Y, \omega, b)$ satisfies $ d^b (\xi \rfloor \omega) = 0$ (the constant $c$ in (\ref{eq:lcpsv}) is zero), equivalently, $\LL_\xi ^b (\omega) = 0$.
\item  Assume that $\omega = d ^b \theta$ for some $\theta \in \Omega ^1 (Y)$.
Then $\xi\rfloor \omega = c\theta  + \alpha$  for some $\alpha \in  \ker  d^b \cap  \Omega^1(Y)$. In this case
$\LL_\xi ^b \omega = c\, \omega$.
\end{enumerate}
\end{cor}

Lemma \ref{lcpsvectorfield} motivates the following definition

\begin{defn}\label{def:ham} A vector field $\xi$ on an  l.c.p-s. manifold  $(Y, \omega, b)$
is called {\it pre-Hamiltonian}, if  $\xi \rfloor \omega = d^b f$ for some $f \in  C^\infty (Y)$.
A diffeomorphism $\phi$ is  called {\it pre-Hamiltonian}, if
it is generated by a time-dependent pre-Hamiltonian  vector field.
\end{defn}

\begin{rem}\label{rem:vectorfield} \begin{enumerate}
\item  If $Y$ is an l.c.s. manifold, our definition of a pre-Hamiltonian vector field coincides
with the definition of a Hamiltonian vector field given by Vaisman \cite[(2.3)]{vaisman:lcs}.
For $b =0$,  our definition of a pre-Hamiltonian vector  field agrees with the
definition in \cite[Definition 3.3]{oh-park}.

\item Clearly, any vector field $\xi$ on $Y$ tangent to $\FF$ is  pre-Hamiltonian  with the Hamiltonian $f = 0$.
 Using Lemma \ref{lcpsvectorfield}  we obtain again the second assertion of
 Proposition \ref{tr-sy-form}, noting that  the corresponding constant $c$ is zero.
 \end{enumerate}
\end{rem}

The following Theorem is an extension of  Theorem 8.1 in \cite{oh-park}.

\begin{thm}\label{extlcps}
Any l.c.p-s.  (resp. pre-Hamiltonian) vector
field $\xi$ on an l.c.p-s. manifold $(Y,\omega, b)$ can be
extended to an l.c.s. (resp. Hamiltonian)
vector field on $(U,\omega_G, \pi^*b)$.
\end{thm}
\begin{proof}
Let  $\xi$  be a  l.c.p-s.  (resp.  pre-Hamiltonian) vector field on $(Y, \omega, b)$.
We decompose
$$\xi = \xi_G + \xi _E$$
where  $\xi _G \in G$ and $\xi_E$ is tangent to $\FF$.
By Remark \ref{rem:vectorfield} (2), $\xi_E$ is  pre-Hamiltonian, hence it suffices to  show that
\begin{enumerate}
\item $\xi_E$  extends to a Hamiltonian vector field on  $(U, \omega_G, \pi^*b)$;
\item $\xi _G$ extends to a l.c.s. (resp.  Hamiltonian) vector field on  $(U, \omega_G, \pi ^*b)$.
\end{enumerate}
To prove (1),  we  define a Hamiltonian function $\widehat f$ on $(U\subset E^*, \omega_G, \pi^*b)$ as follows
$$ \widehat f (\widehat \alpha ) : = \langle  \widehat \alpha, \xi_E \rangle.$$
 It is straightforward to check that
\begin{equation}
 \widehat f |_Y =  0 = f, \text{ and } (d^b \widehat f)|_{Y} = (d\widehat f)| _Y .\label{eq:ext1}
 \end{equation}
Denote by $(d^{\pi ^* b} \widehat f)_{\#\omega_G}$  the  associated Hamiltonian vector field on $U$.
Using (\ref{eq:ext1}), (\ref{eq:fbeta*})  and the coordinate expression of  $\omega_G$  in (\ref{eq:omegaU}), we obtain  easily that
$$\xi_E(y) = (d^{\pi ^* b} \widehat f)_{\#\omega_G}(y)$$
for all $y \in Y$. This  proves (1).

Now  we shall show (2). Since $\omega_G| _\FF = 0$, we have
\begin{equation}
(\xi_G \rfloor \omega_G) (y) = (\xi_G \rfloor \omega ) (y)\label{eq:coincide}
\end{equation}
for all $y \in Y$.
Suppose $[\omega] = 0 \in H ^2 _b (Y)=  H^2_{\pi^* b} (U)$.
Then $\omega_G = d^{\pi^*b} \theta_U$ for some  1-form $\theta_U$ on $Y$. Since
$$ (d^{\pi^* b} \theta _U)|_Y = d ^b (\theta _U)|_Y, $$
the one-form
\begin{equation}
\theta := (\theta _U )| _Y\label{eq:coincide2}
\end{equation}
satisfies the condition in Corollary \ref{cor:lcps} (2).
Using Corollary \ref{cor:lcps} (2), formulas  (\ref{eq:coincide}), (\ref{eq:coincide2})
and the non-degeneracy of $\omega _U$,
we  observe that the extendability of $\xi_G$
is equivalent to the extendability of the one-form  $\alpha$  associated to $\xi_G$ as  in Corollary \ref{cor:lcps} (2)
 to a one-form $\alpha _U$ on $U$
satisfying the following condition:
\begin{equation}
d^{\pi^* b} \alpha _U = 0 \text { and } (\alpha _U)|_Y = \alpha.\label{eq:coincide3}
\end{equation}
Set $\alpha _U: = \pi^*(\alpha)$. Then $\alpha_U$ satisfies (\ref{eq:coincide3}).
This proves  Theorem \ref{extlcps}  for the case $[\omega] = 0 \in H^2 _b (Y)$.

Now assume that $[\omega] \not =  0 \in H ^2 _b (Y)=  H^2_{\pi^* b} (U)$.
By Corollary \ref{cor:lcps}(1)  $d^b (\xi \rfloor \omega) = 0$, or equivalently,
$\xi_G \rfloor \omega = \gamma$, where $d^b\gamma = 0$. Since $\omega_U$ is non-degenerate, using (\ref{eq:coincide}), we note that
the required extendability of $\xi_G$
is equivalent to the extendability  $\gamma$  to a one-form $ \gamma_U$ on $U$  such that $d^{ \pi^*b}\gamma_U = 0$.
Clearly $\gamma_U : = \pi^* (\gamma)$  satisfies the required condition.
This  proves (2) and  completes the proof of Theorem \ref{extlcps}.
\end{proof}

Now we study the geometry of the  master equation for coisotropic sections $s \in E^*$.
By Proposition \ref{prop:coisotropic}, the coisotropic condition for $s$ is given by
\begin{equation}
(\omega - d^b(p_G^* s)) ^{ k+1} = 0\in \Omega ^{2k +2} (Y).\label{eq:main}
\end{equation}

Abbreviate $p_G^* s$ as $s_G$  and $p_G^* (E^*)$ as $G^\circ$. Note that $G^\circ \subset T^*Y$ is  the annihilator of $G$.
We rewrite the master equation (\ref{eq:main}) in the following form
\begin{eqnarray}
s_G \in G^\circ, \nonumber \\
 (\omega - d^b s_G) ^{k+1} = 0.\label{eq:sg}
 \end{eqnarray}

Since $p_G^*|_{E^*} : E^* \to G^ \circ$ is a  bundle isomorphism,
$Y$ is  also  a coisotropic  submanifold in  $(p_G^* (U) \subset G^\circ, (p^*_G|_{E^*}^{-1})^* (\omega_G))$.
Abusing the notation, we  abbreviate $(p^*_G|_{E^*}^{-1})^*(\omega_G)$ as  $\omega _G$  and $p_G^* (U)$ as $U$.   Clearly,
the linearized equation of (\ref{eq:sg}) at the zero section is
\begin{equation}
\omega ^k \wedge d^b \alpha = 0.\label{eq:sglin}
\end{equation}
Since $\omega ^k| _ G \not = 0$ and $\omega ^{k+1} = 0$, the
linearized equation of (\ref{eq:main}) is equivalent to the following  equation
\begin{equation}
 d^{\bar b}_\FF \alpha = 0 \label{eq:mainlin}
\end{equation}
for a section $\alpha \in E^*$. Here $\bar b$ denote the restriction of $b$ to $\FF$.

\begin{defn}\label{def:lcse}  Two sections $s_0, s_1: Y \to U\subset  G^\circ$ are called
{\it Hamiltonian equivalent} (resp. l.c.s. {\it equivalent}), if
there exists  a   family of Hamiltonian   diffeomorphisms (resp. l.c.s. diffeomorphisms)
$\psi _t$ of $(U, \omega_G)$  and
a family of diffeomorphisms $g_t \in Diff (Y)$, $t \in [0,1]$,  such that $g_0 = Id | _{Y}$,
$\psi_0= Id|_{U} $  and
$s_1 = \psi_1 \circ s_0\circ g_1$.

Two sections  $\xi_0, \x_1: Y \to U$  are called {\it  infinitesimally Hamiltonian equivalent}
(resp. {\it infinitesimally l.c.s. equivalent}),
if  $\xi_0 - \xi_1$ is the vertical (fiber) component of  a Hamiltonian
(resp. l.c.s.) vector  field  on $(U, \omega_G)$.
\end{defn}

Clearly, if $s_0$  and $s_1$ are  (Hamiltonian)  equivalent sections, and $s_0$  is
a coisotropic section, then $s_1$ is also a coisotropic section.

\begin{lem}\label{ham=coh} Two solutions  of the linearized equation (\ref{eq:mainlin})
are infinitesimally Hamiltonian equivalent if and only if
they  are  cohomologous  as elements in $\Omega^1_b (Y, \omega)$.  Consequently,
the set of equivalence classes of the infinitesimally Hamiltonian equivalent solutions of the linearized
equations   is $H^1_b (Y, \omega)$.
\end{lem}
\begin{proof} It suffices to prove Lemma \ref{ham=coh} for the case where one of
the two solutions is the zero section.
For $f \in C^\infty (U)$ denote by $(d^{\pi^*b} f)_{\#\omega_G}$  the associated
Hamiltonian  vector field on $U$.  We identify $Y$ with the zero section of $G^\circ\supset U$
and for $ y \in Y$ we denote by $T_y ^{ver}  G^\circ$  the  vertical  (fiber) component
of $T_y  G^\circ = T_y Y \oplus G^\circ_y$.  For any $V \in T_y  G^\circ$  denote by $V^{ver}$
the vertical component of $V$.
Now assume that  a section $\xi: Y \to U\subset G^\circ$   is infinitesimally  Hamiltonian equivalent
to the  zero section, i.e.  there  exists $f \in  C^\infty (U)$ such  that
$\xi = (d ^{\pi ^* b} f)_{\#\omega_G} ^{ver}$. 
Abusing notation we denote by $\pi$ the projection $G^\circ \to Y$.  Using (\ref{eq:fbeta*}) and (\ref{eq:omegaU}), we obtain  for all $y \in Y$
\begin{equation}
(d ^{\pi ^* b} f)_{\#\omega_G}^{ver}(y) = (d ^{\pi ^* b} f |_{ \pi^{-1}(\FF)}) _{f^*_\beta \wedge dp_\beta}^{ver}(y)
=  d^{\bar b}_\FF f(y),\label{eq:vertical}
\end{equation}
where $(d ^{\pi ^* b} f |_{ \pi^{-1}(\FF)}) _{f^*_\beta \wedge dp_\beta}^{ver}(y)$ denotes  the vertical  (fiber) component of   the  vector in  $T_y(\pi^{-1} (\FF)) = E(y) \oplus E^* (y)$ that is dual to the one-form
$d ^{\pi ^* b} f |_{ \pi^{-1}(\FF)}(y)$  with respect to the non-degenerate two-form  $\sum _\beta f^*_\beta \wedge dp_\beta (y)\in \Lambda ^2 T_y ^* (\pi^{-1} (\FF))$.
Hence $[\xi ] = 0 \in H^1_b (Y, \omega)$.

Now assume that $\xi = d^{\bar b}_{\FF} f$ where $f \in \Omega ^ 0 (U)$. By (\ref{eq:vertical})
$\xi$ is infinitesimally Hamiltonian equivalent
to the zero section. This   proves the first assertion of Lemma \ref{ham=coh}. The second
assertion is an immediate    consequence of the first one and (\ref{eq:mainlin}). This
completes the  proof  of Lemma \ref{ham=coh}.
\end{proof}

Next, let us  consider  the case where $\xi$ is infinitesimally l.c.s. equivalent
to the zero section, i.e. there is a l.c.s. vector field $\widehat \xi$ on $U$ such that
$\xi(y)$ is the vertical (fiber) component of $\widehat \xi(y)$  for all $y\in Y$. 
\begin{enumerate}
\item The case with $[\omega_G] \not = 0 \in  H^2 _{\pi^* b} (U, \R)= H^2 _b (Y, \R)$:
Corollary \ref{cor:lcps} (1) implies that   $d^{\bar b}(\widehat \xi \rfloor \omega_G) = 0$. 
The same argument as in the proof of  Lemma \ref{ham=coh}   yields that  for all $y \in Y$
$$\xi(y) = (\widehat \xi \rfloor \omega_G) |_\FF (y)\in  E^*(y).$$
This leads to  specify a subgroup $H^1_{\bar b, ext} (\FF)$ of the group
$H^1 _{\bar b} (\FF)$  whose elements are the restriction of $d^{\pi ^*b}$-closed one-forms  on $U$.
It is easy to see that
$$H^1_{\bar b, ext} (\FF) = i^* (H^1 _b (Y)),$$
where $i : \FF \to  Y$ is the  natural inclusion.

\item The case with $\omega_G = d^b \theta_U$  for some $\theta_U  \in \Omega ^1 (U)$:
Clearly $d ^b _\FF (\theta_U|_\FF) = 0$. Corollary \ref{cor:lcps} (2) implies that
 $(\widehat \xi \rfloor \omega_G)|_{\FF} = c\theta_U |_{\FF} + \alpha |_{\FF}$,
where $\alpha \in \Omega ^1 (U)$ and $d^{\pi^*b}\alpha  = 0$. Using the   argument in the proof of Lemma \ref{ham=coh} we get
$$ \xi(y) = c\theta_U |_{\FF}(y) + \alpha |_{\FF}(y) \in E^*(y).$$
\end{enumerate}

The discussion above yields immediately

\begin{lem}\label{lem:equi} Denote $H^1_b(Y,\omega) = H^1_{\bar b} (\FF)$.
The set of the  infinitesimal l.c.s.  equivalence  classes  of the solutions
$\xi$ of the linearized equation (\ref{eq:mainlin}) has one-one correspondence with
\begin{enumerate}
\item
$
H^1 _b(Y,\omega)/ i^* (H^1 _b (Y))
$
if $[\omega] \not = 0$ in $H^2 _b(Y)$ and
\item
$
H^1 _b(Y,\omega)/(i^* (H^1 _b (Y)) + \langle  \theta |_\FF \rangle _ {\otimes \R})
$
if $\omega = d^b \theta$.
\end{enumerate}

\end{lem}







\section{Geometry of the l.c.s. thickening of a  l.c.p-s. manifold}
\label{sec:connections}

In this section, 
imitating the scheme performed for the pre-symplectic manifolds in \cite{oh-park}, we introduce  special coordinates
in the  l.c.s. thickening $(U,\omega_U,\pi^*b)$  of a  l.c.p-s. manifold and we compute
 important geometric structures in these coordinates ((\ref{eq:thetaG}), (\ref{eq:omegaU}), (\ref{eq:Tpi-1FFomegaU})),
 preparing  for  our study of  deformations of compact coisotropic submanifolds in l.c.s. manifolds in the next two sections. 

Again we start with a splitting
$$
TY = G \oplus E,
$$
the associated bundle projection $\Pi: TY \to TY$,  the associated
canonical one form $\theta_G$, and the l.c.s. form
\be\label{eq:omegaE*}
\omega_{U} = \pi^*\omega - d^{\pi^*b} \theta_G
\ee
on $U \subset E^*$. Let
$$
(y^1, \cdots, y^{2k}, q^1, \cdots, q^{n-k})
$$
be coordinates on $Y$ adapted to the null foliation on an open
subset $V \subset Y$. By choosing the frame
$$
\{f_1^*, \cdots, f_{n-k}^*\}
$$
of $E^*$ that is dual to the frame $\{\frac{\del}{\del q^1}, \cdots,
\frac{\del}{\del q^{n-k}}\}$ of $E$,
we introduce the {\it canonical coordinates} on $E^*$ by writing
an element $\widehat\alpha \in E^*$ as a linear combination of $\{f_1^*,
\cdots, f_{n-k}^*\}$
$$
\widehat\alpha = p_\beta f_\beta^*,
$$
and taking
$$
(y^1, \cdots, y^{2k}, q^1, \cdots, q^{n-k}, p_1, \cdots, p_{n-k})
$$
as the associated coordinates.

For a given splitting $\Pi: TY = G\oplus T\FF$,  there exists the
unique splitting of $TU$
\begin{equation}\label{eq:TUtildeG}
TU = G^\sharp \oplus T\pi^{-1}(\FF)
\end{equation}
that satisfies
\begin{equation}\label{eq:tildeG}
G^\sharp = (T_{\widehat \alpha}\pi^{-1}(\FF))^{\omega_U}
\end{equation}
for any $\widehat \alpha \in U$, which is invariant under the action of
l.c.s. diffeomorphisms on $(U, \omega_U, \pi^* b)$ that preserve the
leaves of $\pi^{-1}(\FF)$.

\begin{defn}\label{transymp}
We call the above unique splitting the {\it leafwise l.c.s.
connection of $U \to Y$ compatible to the splitting $\Pi: TY = G
\oplus T\FF$} or simply a {\it leafwise l.c.s.
$\Pi$-connection} of $U \to Y$.
\end{defn}
We would like to emphasize that this connection is not
a vector bundle connection of $E^*$ although $U$ is a subset of
$E^*$, which reflects {\it nonlinearity} of this connection.
We refer the reader to subsection \ref{subsec:splittings} for more  detailed explanation.

Note that the splitting $\Pi$ naturally induces the splitting
$$
\Pi_*: T^*Y = (T\FF)^\circ \oplus G^\circ.
$$
For the given splitting  $TY= G \oplus E$ we can write, as in \cite[(4.5)]{oh-park},
$$
G_x = \mathrm{span} \left\{{\partial \over \partial y^ i } +
\sum _{ \alpha =1} ^{m -l} R^\alpha _i {\partial \over \partial q ^\alpha}\right\}_{1\le i \le l}
$$
for   some $R^\alpha _i $s, which  are uniquely determined by the splitting and the given coordinates.

To derive the coordinate expression
of $\theta_G$, we compute
\begin{eqnarray*}
\theta_G\Big(\frac{\del}{\del y^i}\Big) & = & \widehat\alpha\Big(p_G\circ
T\pi(\frac{\del}{\del y^i})\Big)
= \widehat\alpha\Big(p_G(\frac{\del}{\del y^i})\Big) \\
& = & p_\beta f_\beta^*\Big(-R_i^\alpha \frac{\del}{\del q^\alpha}\Big) = -
p_\alpha
R_i^\alpha, \\
\theta_G\Big(\frac{\del}{\del q^\beta}\Big)  & = & p_\beta, \qquad
\theta_G\Big(\frac{\del}{\del p_\beta}\Big) = 0.
\end{eqnarray*}
Hence we derive
\begin{equation}\label{eq:thetaG}
\theta_G = p_\beta(dq^\beta - R_i^\beta dy^i).
\end{equation}
Here we note that
$$
(dq^\beta - R^\beta_idy^i)|_{G_x} \equiv 0.
$$
This shows that if we identify $E^* = T^*\FF$ with $G^\circ$,
then we may
write the dual frame on $T^*\FF$ as
\begin{equation}\label{eq:fbeta*}
f_\beta^* = dq^\beta - R^\beta_idy^i.
\end{equation}
Motivated by this, we write
\bea\label{eq:dthetaG}
d\theta_G & = & dp_\beta \wedge (dq^\beta - R_i^\beta dy^i)
- p_\beta dR_i^\beta \wedge dy^i
\nonumber \\
& = & dp_\beta \wedge (dq^\beta - R_i^\beta dy^i) - (dq^\gamma - R^\gamma_jdy^j) \wedge
p_\beta\frac{\del R^\beta_i}{\del q^\gamma} dy^i\nonumber \\
&{}& \quad - p_\beta \left(\frac{\del R_i^\beta}{\del y^j}
- R^\gamma_j \frac{\del R_i^\beta}{\del q^\gamma}\right) \, dy^j \wedge dy^i
\eea
and
\be\label{eq:bthetaG}
\pi^*b \wedge \theta_G = (b_\gamma dq^\gamma + b_i dy^i) \wedge p_\delta(dq^\delta - R_j^\delta dy^j).
\ee
Combining \eqref{eq:dthetaG}, \eqref{eq:bthetaG} and \eqref{eq:piomegaY}, we derive
\bea\label{eq:omegaU}
\omega_U
& = & \frac{1}{2}\Big(\omega_{ij} - p_\beta F^\beta_{ij} \Big)
dy^i \wedge dy^j \nonumber \\
& \quad &  - \left(dp_\nu + p_\nu(b_\gamma dq^\gamma +b_idy^i) + p_\beta\left(\frac{\del R_i^\beta}{\del
q^\nu}\right) dy^i\right) \wedge (dq^\nu- R_j^\nu dy^j)\nonumber\\
& = & \frac{1}{2}\Big(\omega_{ij} - p_\beta F^\beta_{ij} \Big)
dy^i \wedge dy^j \nonumber \\
& \quad &  - \left(dp_\nu + p_\nu b_\gamma \left(dq^\gamma - R^\gamma_idy^i\right)
+\left(p_\nu(b_i + b_\gamma R^\gamma_i)+ p_\beta \frac{\del R_i^\beta}{\del
q^\nu}\right) dy^i\right)\nonumber \\
&{}& \quad \wedge (dq^\nu- R_j^\nu dy^j)
\eea
similarly as in the derivation of  \cite[(6.8)]{oh-park},
where $F^\beta_{ij}$ are the components of the transverse
$\Pi$-curvature of the null-foliation given by \eqref{eq:Fbetaij} in the Appendix.

Note that we have
$$
T\pi^{-1}(\FF) = \operatorname{span} \Big\{\frac{\del}{\del q^1},
\cdots, \frac{\del}{\del q^{n-k}},  \frac{\del}{\del p_1}, \cdots,
\frac{\del}{\del p_{n-k}} \Big\}
$$
which is independent of the choice of the above induced foliation
coordinates of $TU$.

Now we compute $G^\sharp =(T\pi^{-1}(\FF))^{\omega_U}$ in $TU$ in terms of
these induced foliation coordinates. We will determine when the
expression
$$
a^j(\frac{\del}{\del y^j} + R_j^\alpha\frac{\del}{\del q^\alpha})
+ d^\beta \frac{\del}{\del q^\beta} + c_\gamma \frac{\del}{\del
p_\gamma}
$$
satisfies
$$
\omega_U\Big( a^j(\frac{\del}{\del y^j} +
R_j^\alpha\frac{\del}{\del q^\alpha}) + d^\beta \frac{\del}{\del
q^\beta}+ c_\gamma \frac{\del}{\del p_\gamma},
T\pi^{-1}(\FF)\Big)= 0.
$$
It is immediate to see by
pairing with $\frac{\del}{\del p_\mu}$
\begin{equation}\label{eq:bbeta}
d^\beta = 0, \quad \beta = 1, \cdots, n-k.
\end{equation}
Next we study the equation
$$
0 = \omega_U\Big (a^j(\frac{\del}{\del y^j} +
R_j^\alpha\frac{\del}{\del q^\alpha}) + c_\gamma \frac{\del}{\del
p_\gamma}, \frac{\del}{\del q^\nu}\Big)
$$
for all $\nu = 1, \cdots, n-k$. A straightforward check provides
\begin{equation}\label{eq:cnu}
a^j\left(p_\nu(b_i + b_\gamma R^\gamma_i)+ p_\beta \frac{\del R_i^\beta}{\del
q^\nu}\right) + c_\nu = 0
\end{equation}
for all $\nu$ and $j$. Combining (\ref{eq:bbeta}) and
(\ref{eq:cnu}), we have obtained
\begin{equation}\label{eq:Tpi-1FFomegaU}
(T\pi^{-1}(\FF))^{\omega_U} = \operatorname{span} \left\{
\frac{\del}{\del y^j} + R_j^\alpha \frac{\del}{\del q^\alpha} -
\left(p_\nu(b_i + b_\gamma R^\gamma_i)+ p_\beta \frac{\del R_i^\beta}{\del
q^\nu}\right) \frac{\del}{\del p_\nu}
\right\} _{1 \leq j \leq 2k}.
\end{equation}

\begin{rem} Just as we have been considering $\Pi: TY = G \oplus T\FF$ as a
``connection'' over the leaf space, we may consider the splitting
$\Pi^\sharp: TU = G^\sharp \oplus T(\pi^{-1}\FF)$ as the leaf
space connection {\it canonically induced from $\Pi$ under the
fiber-preserving map}
$$
\pi: U \to Y
$$
over the same leaf space $Y/\sim$: Note that the space of leaves
of $\FF$ and $\pi^{-1}\FF$ are canonically homeomorphic.
\end{rem}

\section{Master equation in coordinates}
\label{sec:master}

We will  derive the first equation for  the graph of a section $s: Y
\to E^*\cong NY$ to be coisotropic with respect to $\omega_U$ (Theorem \ref{thm:2}), which is a natural extension of a similar equation
in the  symplectic setting obtained in \cite{oh-park}. We
call the corresponding equation the {\it classical part} of the
master equation for the deformation theory of coisotropic submanifolds.

Recall that an Ehresmann connection of $U \to Y$ with a structure
group $H$ is a splitting of the exact sequence
$$
0 \to VTU \longrightarrow TU \stackrel{T\pi} \longrightarrow TY
\to 0
$$
that is invariant under the action of the group $H$. Here $H$ is
not necessarily a finite dimensional Lie group. In other words, an
Ehresmann connection is a choice of decomposition
$$
TU = HTU \oplus VTU
$$
that is invariant under the fiberwise action of $H$. Recalling
that there is a canonical identification $V_{\widehat\alpha} TU \cong
V_{\widehat\alpha }TE^* \cong E^*_{\pi(  \widehat\alpha  )}$, a connection can be
described as a {\it horizontal} lifting $HT_{\widehat \alpha }U$ of $TY$ to
$TU$ at each point $y \in Y$ and $\widehat\alpha \in U \subset E^*$ with
$\pi(\widehat\alpha) = y$.  We denote by $F^\# \subset HTU$ the horizontal
lifting of a subbundle $F \subset TY$ in general.

Let $(y^1,\cdots, y^{2k},q^1,\cdots, q^{n-k})$ be  foliation
coordinates of $\FF$ on $Y$ and
$$
(y^1,\cdots, y^{2k},q^1,\cdots, q^{n-k}, p_1, \cdots, p_{n-k})
$$
be the induced foliation coordinates of $\pi^{-1}(\FF)$ on $U$.
Then $G^\# = (T\pi^{-1}(\FF))^{\omega_U}$ has the natural basis
given by
\be\label{eq:ej}
e_j =
\frac{\del}{\del y^j} + R_j^\alpha \frac{\del}{\del q^\alpha} -
\left(p_\nu(b_i + b_\gamma R^\gamma_i)+ p_\beta \frac{\del R_i^\beta}{\del
q^\nu}\right)
\frac{\del}{\del p_\nu}
\ee
which are basic vector fields of $T(\pi^{-1}\FF)$. We also
denote
$$
f_\alpha = \frac{\del}{\del q^\alpha}.
$$
We define a local lifting of $E$
\begin{equation}\label{eq:Esharp}
E^\sharp = \operatorname{span} \Big\{f_1, \cdots, f_{n-k}\Big\}.
\end{equation}
The lifting \eqref{eq:Esharp} of $E$ provides a local splitting
$$
TU = (G^\sharp \oplus E^\sharp) \oplus VTU \to TY
$$
and defines a locally defined Ehresmann connection  on  the bundle $ U \to Y $.
From the
expression (\ref{eq:omegaU}) of $\omega_U$, it follows that
$G^\sharp \oplus E^\sharp$ is a coisotropic lifting of $TY$ to $T
U$. We denote by $\Pi^v: TU \to VTU$ the vertical projection with
respect to this splitting.

With this preparation, we are finally ready to derive the master
equation. Let $s: Y \to U \subset E^*$ be a section and denote
\begin{equation}\label{eq:nablas}
\nabla s: = \Pi^v\circ ds
\end{equation}
its locally defined ``covariant derivative''. In coordinates
$(y^1, \cdots, y^{2k},q^1, \cdots, q^{n-k})$, we have
\begin{eqnarray*}
ds\Big(\frac{\del}{\del y^j}\Big) & = &\frac{\del}{\del y^j} +
\frac{\del s_\alpha}{\del y^j} \frac{\del}{\del p_\alpha} \\
& = & e_j - R^\alpha_j\frac{\del}{\del q^\alpha} + \left(\frac{\del s_\nu}{\del y^j} +
s_\nu(b_i + b_\gamma R^\gamma_i)+ s_\beta \frac{\del R_i^\beta}{\del
q^\nu}\right)\frac{\del}{\del p_\nu}.
\end{eqnarray*}
Therefore we have derived
\begin{equation}\label{eq:nablas}
\nabla s\Big(\frac{\del}{\del y^j}\Big) = \left(\frac{\del s_\nu}{\del y^j} +
s_\nu(b_i + b_\gamma R^\gamma_i)+ s_\beta \frac{\del R_i^\beta}{\del
q^\nu} \right)\frac{\del}{\del p_\nu}.
\end{equation}
Similarly we compute
\begin{eqnarray*}
ds\Big(\frac{\del}{\del q^\nu}\Big) & = &\frac{\del}{\del q^\nu} +
\frac{\del s_\alpha}{\del q^\nu}\frac{\del}{\del p_\alpha} \\
& = & \frac{\del}{\del q^\nu} +
\frac{\del s_\alpha}{\del q^\nu} \frac{\del}{\del p_\alpha},
\end{eqnarray*}
and so
\begin{equation}\label{eq:nablanus}
\nabla s\Big(\frac{\del}{\del q^\nu}\Big) =
\frac{\del s_\alpha}{\del q^\nu} \frac{\del}{\del p_\alpha}.
\end{equation}

Recalling that $T_{\widehat\alpha} U = (E_{\widehat\alpha}^\# \oplus VT_{\widehat\alpha}
U)^{\omega_U} \oplus E_{\widehat\alpha} ^\# \oplus VT_{\widehat\alpha} U$, we conclude
that the graph of $ds$ with respect to the frame
$$
\Big\{e_1, \cdots, e_{2k}, f_1, \cdots, f_{n-k}, \frac{\del}{\del
p_1}, \cdots, \frac{\del}{\del p_{n-k}} \Big\}
$$
can be expressed by the linear map
\begin{eqnarray*}
A_H: (E^\#\oplus VTU)^{\omega_U}
& \to & VTU \cong E^*;\quad (A_H)_\alpha^i = \nabla_i s_\alpha, \\
A_I: E^\# & \to &VTU \cong E^*;  \quad (A_I)_\alpha^\beta =
\nabla_\beta s_\alpha,
\end{eqnarray*}
where
\bea\label{eq:nablaincoor}
\nabla s\Big(\frac{\del}{\del y^i}\Big) & = & (\nabla_i s_\alpha)
\frac{\del}{\del q^\alpha}, \quad  \nabla_i s_\alpha
: =  \frac{\del s_\alpha}{\del y^j}
+s_\alpha(b_i + b_\gamma R^\gamma_i)+ s_\beta \frac{\del R_i^\beta}{\del
q^\alpha} , \nonumber\\
\nabla s\Big(\frac{\del}{\del q^\beta}\Big)  & = & (\nabla_\beta s_\alpha)
\frac{\del}{\del q^\alpha}, \quad
\nabla_\beta s_\alpha  : =  \frac{\del s_\alpha}{\del q^\beta},
\eea

Finally we note that
$$
\omega_U(s)(e_i,e_j) = w_{ij} - s_\beta F^\beta_{ij} : =
\widetilde \omega_{ij}
$$
and denote its inverse by $(\widetilde \omega^{ij})$. Note that
$(\widetilde \omega_{ij})$ is invertible if $s_\beta$ is
sufficiently small, i.e., if the section $s$ is $C^0$-close to the
zero section, or its image stays inside of $U$.
Now Proposition 2.2 immediately implies

\begin{thm}\label{thm:2} Let $\nabla s$ be the vertical projection of $ds$
as in (\ref{eq:nablas}). Then  the graph of the section $s: Y \to
U$ is coisotropic with respect to $\omega_U$ if and only if $s$
satisfies
\begin{equation}
\nabla_i s_\alpha \widetilde \omega^{ij} \nabla_j s_\beta
= \nabla_\beta s_\alpha - \nabla_\alpha s_\beta
\end{equation}
for all $\alpha > \beta$
or
\begin{equation}\label{eq:masterincoord}
\frac{1}{2}
(\nabla_i s_\alpha \widetilde \omega^{ij} \nabla_j s_\beta)
f_\alpha^* \wedge f_\beta^* =
(\nabla_\beta s_\alpha)  f_\alpha^* \wedge f_\beta^*
\end{equation}
where $f_\alpha^*$ is the dual frame of
$\{\frac{\del}{\del q^1}, \cdots, \frac{\del}{\del q^{n-k}}\}$
defined by (\ref{eq:fbeta*}).
\end{thm}

Note that (\ref{eq:masterincoord}) involves terms of all order of
$s_\beta$ because the matrix $(\widetilde \omega^{ij})$
is the inverse of the matrix
$$
\widetilde \omega_{ij} = \omega_{ij} - s_\beta F^\beta_{ij}.
$$
There is a special case where the curvature vanishes
i.e., satisfies
\begin{equation}\label{eq:special}
F_G = F^\beta_{ij}\frac{\del}{\del q^\beta}\otimes dy^i \wedge
dy^j= 0
\end{equation}
in addition to (\ref{eq:leaves}). In this case, $\widetilde
\omega_{ij} = \omega_{ij}$ which depends only on $y^i$'s and so
does $\omega^{ij}$. Therefore (\ref{eq:masterincoord}) is reduced
to the quadratic equation
\begin{equation}\label{eq:quad-master}
\frac{1}{2}(\nabla_i s_\alpha \omega^{ij} \nabla_j s_\beta)
f_\alpha^* \wedge f_\beta^* = (\nabla_\beta s_\alpha) f_\alpha^*
\wedge f_\beta^*.
\end{equation}

\section{Deformation of strong homotopy Lie algebroids}
\label{sec:deform-shla}

In this section, we provide an invariant description of the master equation we have
derived in the previous section. This can be regarded as the equation for
the coisotropic submanifolds in the formal power series
version of the equation or in the formal manifold in the sense of
Kontsevich \cite{konts:poisson}, \cite{aksz}.

\subsection{$b$-deformed Oh-Park's strong homotopy Lie algebroids \cite{oh-park}}
\label{subsec:b-deformation}

We start with the normal form \eqref{eq:omegaE*} of the symplectic thickening.
We also note that the discussion of leaf space connection and
the curvature, in particular the one-form $\theta_G$
does not depend on the closed one-form $b$ but only
depends on the conformal pre-symplectic form $\omega$ and the splitting
$TY = G\oplus E$ only.
In this regard, we can view the normal form in \eqref{eq:omegaE*} as
a deformation of the nondegenerate two form
$$
\pi^*\omega - d\theta_G
$$
to a conformal symplectic form relative to $\pi^*b$.
So from now on, we denote $\omega_{U} = \pi^*\omega - d\theta_G$ and
\be\label{eq:omegaE*b}
\omega_{U}^b = \omega_{U} - d^{\pi^*b} \theta_G.
\ee
This deformation is responsible for the appearance of $b$-terms in \eqref{eq:omegaU}
and then \eqref{eq:Tpi-1FFomegaU}, \eqref{eq:ej} and eventually for
the covariant derivative \eqref{eq:nablas}.

Again we regard \eqref{eq:nablaincoor} as the deformation of
the old covariant derivative formula appearing in  \cite[(7.3)]{oh-park} and denote
the full covariant derivative
\be\label{eq:nablabs}
\nabla^b s = \nabla s + b|_\FF s + \langle b|R|s
\ee
where $b|_\FF$ is the restriction to the null-foliation of the one-form $b$
and $\langle b|R|$ is the pairing of $b$ and $R$ which produces
a one-form on $G$ with values in $T\FF$.

Now we give a deformed version  of the notion of {\it strong
homotopy Lie algebroid} introduced in \cite{oh-park}.

\begin{defn} Let $E \to Y$ be a Lie algebroid. A  $b$-deformed
{\it $L_\infty$-structure} over the
Lie algebroid is a structure of strong homotopy Lie algebra
$(\frak l[1], \frak m)$ on the associated $b$-deformed $E$-de Rham complex
$\frak l^\bullet = \Omega^\bullet(E)
= \Gamma(\Lambda^\bullet(E^*))$
such that $\frak m_1$ is the $E$-differential $^Ed^{\bar b}$ induced by the (deformed) Lie algebroid structure on $E$
as described in subsection \ref{sliealgebroid}.
We call the pair $(E \to Y, \frak m)$ a $b$-deformed strong homotopy  Lie algebroid.
\end{defn}

Here we refer to \cite{nest-tsygan} or Appendix \ref{sliealgebroid} for the
definition of $E$-differential used in this definition.

With this definition of  a $b$-deformed  strong homotopy  Lie algebroid, we will
show that for given l.c.p-s. manifold $(Y,\omega,b)$
each splitting $\Pi: TY = G\oplus T\FF$ induces a
canonical $L_\infty$-structure over the Lie algebroid $T\FF \to Y$.

The following linear map and quadratic map are introduced in
\cite{oh-park} which play crucial roles in the construction of
$L_\infty$-structure on the foliation de Rham complex: a linear map
\begin{equation} \label{eq:lin}
\widetilde \omega: \Omega^1(Y;\Lambda^\bullet E^*) \to
\Gamma(\Lambda^{\bullet +1}E^*) = \Omega^{\bullet + 1}(\FF),
\end{equation}
a quadratic map
\be\label{eq:quad} \langle \cdot, \cdot
\rangle_{\omega}: \Omega^1(Y;\Lambda^{\ell_1} E^*) \otimes
\Omega^1(Y;\Lambda^{\ell_2} E^*) \to \Omega^{\ell_1 +
\ell_2}(\FF),
\ee
and the third map that is induced by the transverse $\Pi$-curvature, whose
definition is now in order.

We recall the definitions of those maps. The linear map $\widetilde \omega$ is
defined by
\begin{equation}
\label{eq:lin} \widetilde\omega(A): = (A|_E)_{skew}.
\end{equation}
Here note that an element $A \in \Omega^1(Y; \Lambda^kE^*)$ is a
section of $T^*Y\otimes \Lambda^k E^*$. Restricting $A$ to $E$ for
the first factor  we get  $A|_E \in E^* \otimes \Lambda^k E^*$. Then $(A|_E)_{skew}$ is the
skew-symmetrization of $A|_E$. The quadratic map is defined by
$$
\langle A,B \rangle_{\omega}:= \langle A|\pi |B \rangle - \langle
B|\pi|A \rangle
$$
where $\pi$ is the transverse Poisson bi-vector on $N^*\FF$ associated to the
transverse symplectic form $\omega$ on $N\FF$.

We will denote
\begin{equation}\label{eq:Fsharp}
F^\#:= F\omega^{-1} = F^{\alpha j}_i dy^i \otimes
\Big(\frac{\del}{\del y^j} + R^\beta_j \frac{\del}{\del
q^\beta}\Big) \otimes \frac{\del}{\del q^\alpha} \in \Gamma(G^*
\otimes G \otimes E),
\end{equation}
where $F^{\alpha j}_i = F^\alpha_{ik}\omega^{kj}$. Note that we
can identify $\Gamma(G^* \otimes G \otimes E)$ with
$\Gamma(N^*\FF\otimes N\FF \otimes E)$ via the isomorphism $\pi_G:
G \to N\FF$.

For given $\xi \in \Omega^\ell(\FF)$, we define
\begin{equation}\label{eq:b-lin}
d_{\FF}^b(\xi): = (\nabla^b \xi|_E)_{skew},
\end{equation}
and deformed bracket
\begin{equation}\label{eq:b-quad}
\{\xi_1, \xi_2\}_{\Pi}^b: = \langle \nabla^b \xi_1, \nabla^b \xi_2 \rangle_{\omega} =
\sum_{i < j} \omega^{ij}(\nabla_i^b \xi_1) \wedge (\nabla_j^b \xi_2 ).
\end{equation}
Here the map in \eqref{eq:b-lin} is nothing but the ${\bar b}$-deformed leafwise differential of the
null foliation which is indeed independent of the choice of
splitting $\Pi: TY = G \oplus T\FF$ but depends only on the
foliation and the projection of the one-form $b$ to $\FF$, see also subsection \ref{sliealgebroid}.
By  Remark \ref{rem:onb}  the  obtained leafwise differential  depends only on $\omega$. We
use $d^b_\FF$ and $d^{\bar b}_\FF$ interchangeably.

The second is a bracket
in the transverse direction which is a $b$-deformation of the one
given in  \cite[(9.13)]{oh-park}.

Now we promote the maps $d_\FF^b$ and $\{\cdot, \cdot\}^b$ to an infinite family of
graded multilinear maps
\begin{equation}\label{eq:frakm}
\frak m_\ell^b = (\Omega[1]^\bullet(\FF))^{\otimes\ell} \to
\Omega[1]^{\bullet}(\FF)
\end{equation}
so that the structure
$$
\left(\bigoplus_{j=0}^{n-k} \frak l[1]^j; \{\frak m_\ell^b\}_{1 \leq
\ell < \infty}\right)
$$
defines a {\it strong homotopy Lie algebroid} on $E=T\FF \to Y$ in
the above sense.  Here  $\Omega[1]^\bullet(\FF)$ is the shifted complex of
$\Omega^\bullet(\FF)$, i.e., $\Omega[1]^k(\FF) = \Omega^{k+1}(\FF)$
and $\frak m_1$ is defined by
$$
\frak m_1^b(\xi) = (-1)^{|\xi|} d_{\FF}^b(\xi)
$$
and $\frak m_2$ is given by
$$
\frak m_2^b(\xi_1, \xi_2) =
(-1)^{|\xi_1|(|\xi_2|+1)}\{\xi_1,\xi_2\}_\Pi^b.
$$
On the un-shifted group $\frak l$, $d_{\FF}^b$ defines a
differential of degree 1 and $\{\cdot, \cdot\}_{\omega}$ is a
graded bracket of degree 0 and $\frak m_\ell^b$ is a map of degree
$2-\ell$.

We now define $\frak m_\ell^b$ for $\ell \geq 3$. Here enters the
transverse $\Pi$-curvature $F=F_\Pi$ of the splitting $\Pi$ of the
null foliation $\FF$. We define
\begin{equation}\label{eq:mell}
\frak m_\ell^b(\xi_1, \cdots, \xi_\ell) : =
\sum_{\sigma \in S_\ell}
(-1)^{|\sigma|}\langle \nabla^b \xi_{\sigma(1)},
(F^\#\rfloor \xi_{\sigma(2)}) \cdots
(F^\#\rfloor \xi_{\sigma(\ell-1)})
\nabla^b \xi_{\sigma(\ell)}\rangle_\omega
\end{equation}
where $|\sigma|$ is the standard Koszul sign in the suspended complex.
We have now arrived at our definition of strong
homotopy Lie algebroid associated to the coisotropic submanifolds,
which is a $b$-deformation of the one introduced in \cite[section 9]{oh-park},
but which is applied \emph{after enlarging our category} to that of
locally conformal pre-symplectic two forms instead of pre-symplectic two forms.

\begin{thm}\label{algebroid}
Let $(Y,\omega,b)$ be a l.c.p.s. manifold and $\Pi: TY =
G\oplus T\FF$ be a splitting. Then $\Pi$ canonically induces a
structure of strong homotopy Lie algebroid on $T\FF$ in that the
graded complex
$$
\left(\bigoplus_\bullet \Omega[1]^\bullet(\FF),
\{\frak m_\ell^b \}_{1 \leq \ell < \infty}\right)
$$
defines the structure of strong homotopy Lie algebra. We denote by
$\frak l_{(Y, \omega, b;\Pi)}$ the corresponding strong
homotopy Lie algebra.
\end{thm}
\begin{proof} The proof of this theorem follows the strategy used in
the proof of Theorem 9.4 \cite{oh-park}, which uses the formalism
of super-manifolds and odd symplectic structure on the
super tangent bundle $T[1]U$ \cite{aksz} of the l.c.p.s. thickening $U$ of
$(Y,\omega, b)$.

We change the parity of $TU$ along the fiber
and denote by $T[1]U$ the corresponding super tangent bundle of
$U$. One considers a multi-vector field on $U$ as a (fiberwise)
polynomial function on $T^*[1]U$. For example, the bi-vector field $P$,
inverse to the {\it non-degenerate} form $\omega_U$ (cf. (\ref{eq:omegaE*b})),
defines  a quadratic function, which we denote by $H$.
This also coincides with  the push-forward of the even function $H^*: T[1]U \to \R$ induced by   $\omega _U$.
We denote by $\{\cdot, \cdot\}_\Omega$
the (super-)Poisson bracket associated to the odd symplectic form
$\Omega$ on $T[1]U$. Then the bracket operation
$$
Q : = \{H^*, \cdot\}_\Omega
$$
defines a derivation on the set $\mathcal O_{T[1]U}$ of ``functions''
on $T[1]U$: Here $\mathcal O_{T[1]U}$ is the set of differential
forms on $U$ considered as fiberwise polynomial functions on
$T[1]U$. We refer to \cite{getzler} or  \cite[Appendix]{oh-park} for the precise mathematical meaning for
this correspondence. Therefore it defines an odd vector field.

Restricting ourselves to a Darboux neighborhood of $\L = T\FF[1]
\subset T[1]U$, we identify the neighborhood with a neighborhood
of the zero section $T^*[1]\L$. Using the fact that
(\ref{eq:delta'}) depends only on $\xi$, not on the extension, we
will make a convenient choice of coordinates to write $H$ in the
Darboux neighborhood and describe how the derivation $Q = \{H,
\cdot \}_\Omega$ acts on $\Omega^{\bullet}(\FF)$ in the canonical
coordinates of $T^*[1]\L$. In this way, we can apply the canonical
quantization which provides a canonical correspondence between
functions on ``the phase space'' $T^*[1]\L$ and the corresponding
operators acting on the functions on the ``configuration space''
$\L$, later when we find out how the deformed differential $\delta^b$
\eqref{eq:deltab} acts on $\Omega^\bullet(\FF)$.

We denote by $(y^i,q^\alpha, p_\alpha, y^*_i, q^*_\alpha,
p^\alpha_*)$ the canonical coordinates $T^*\L$ associated with the
coordinates $(y^i,q^\alpha, p_\alpha)$ of $N^*\FF$. Note that
these coordinates are nothing but the canonical coordinates of
$N^*Y \subset T^*U$ pulled-back to $T\FF \subset TU$ and its
Darboux neighborhood, with the corresponding parity change: We
denote the (super) canonical coordinates of $T^*[1]\L$ associated
with $(y^i,q^\alpha\mid p_\alpha)$ by
$$
\Big(\begin{matrix}y^i, & q^\alpha  && \mid p_*^\alpha \\
y^*_i, & q^*_\alpha && \mid p_\alpha
\end{matrix}\Big)
$$
Here we note that the degree of $y^i,\, q^\alpha$ and $p_\alpha$
are $0$ while their anti-fields, i.e., those with $*$ in them have
degree $1$. And we want to emphasize that $\L$ is given by the
equation
\begin{equation}\label{eq:LL}
y_i^* = p_\alpha = p_*^\alpha = 0
\end{equation}
and $(y^i, y^*_i)$, $(p_\alpha, q_\alpha^*)$ and $(p_*^\alpha,
q^\alpha)$ are conjugate variables. In terms of these coordinates,
\cite[(9.23)]{oh-park} provides the formula
\be\label{eq:H^*}
H^* = \frac{1}{2} \widetilde \omega^{ij} y_i^\#y_j^\# + p_*^\delta q_*^\delta.
\ee
Here,  we define $y_i ^\#$ to be
$$
y_i ^\#:= y_i^* + R ^\delta_i p^\delta_* - \left(p_\nu(b_i + b_\gamma R_i^\gamma)
+ p_\beta {\partial R^\beta_i\over \partial q^\delta} q^* _\delta\right)
$$
arising from \eqref{eq:ej} similarly as in \cite[(9.23)]{oh-park}.
When $\omega$ is a closed symplectic form as in \cite{oh-park}, we have
$\{H^*,H^*\} = 0$. However in the current l.c.s. case,
this is no longer the case.

\begin{lem} Consider the l.c.p.s. manifold $(Y,\omega,b)$
and let $(U,\omega_U, \pi^*b)$ be its c.p.s. thickening constructed before.
Regard the closed one-form $\pi^*b$ as the odd function
$b = b_{q^\gamma} p^\gamma_* + b_{p_\delta} q_\delta^* + b_i \psi_i^\#$
where $p^\gamma_* = dq^\gamma, \, \psi_i^\# = dy_i^\#, \, q_\delta^* = dp_\delta$.
Then $\{H^*,H^*\} = b H^*$.
\end{lem}
\begin{proof} In this proof, we use Einstein's summation convention, whenever we feel
convenient. Using the canonical bracket relations,
we compute
\beastar
\{H^*, H^*\} = \frac{1}{2}\left(\frac{\del \widetilde \omega^{ij}}{\del q_\delta} q_\delta^*
+ \frac{\del \widetilde \omega^{ij}}{\del p^\delta} p^\delta_*\right) y_i^\# y_j^\#
 + \frac{\del \widetilde \omega^{ij}}{\del y_k^\#}  y_i^\# y_j^\# \psi_k^\#.
\eeastar
However the $d^b$-closedness of $\omega_U$ means that this sum is  precisely equivalent to
\beastar
(b_{q^\gamma} q^\gamma_* + b_{p_\delta} p_\delta^* \omega^{ij}
+  b_i \psi_i^\#) \frac{1}{2} \omega^{ij}  y_i^\# y_j^\#
= b H^*.
\eeastar
This finishes the proof.
\end{proof}

We will be interested in whether
one can canonically restrict the vector field $Q$ to $\L = T\FF[1]$ or equivalently whether the
function $H$ has constant value on $\L$. Here comes the
coisotropic condition naturally.  The following was proved in
\cite{oh-park}, which still holds for the current l.c.s. case.

\begin{lem} \cite[Lemma 9.5]{oh-park}\label{lem:brane}
Let $H$ be the even function on $T[1]X$ induced by the symplectic
form $\omega_X$, and $H^*: T^*[1] X \to \R$ be its push-forward by
the isomorphism $\widetilde \omega_X: T[1]X \to T^*[1]X$. When $Y
\subset (X,\omega)$ is a coisotropic submanifold we have
$H^*|_{N^*[1]Y} = 0$. Conversely, any (conic) Lagrangian subspace
$\mathbb{L}^* \subset T^*[1]X $ satisfying $H^*|_{\mathbb{L}^*}=0$
is equivalent to $N^*[1]Y$, for some coisotropic submanifold $Y$ of $(X,\omega)$.
\end{lem}

Now we consider the $b$-deformed Hamiltonian
$H^*_b: T[1]U \to \R$ the even function induced by the l.c.s. form
$\omega_U^b$. From the expression \eqref{eq:omegaU}, a straightforward calculation leads to
\be\label{eq:H*b}
H^*_b = H^* + p_\nu q_\nu^* b
\ee
where $q_\nu^* = \chi^\nu - R_j^\nu \psi_j$ is the conjugate variable of
$p_\nu$ and $b = b_\gamma \chi^\gamma + b_i \psi^i$.

\begin{prop}\label{prop:H*bH*b} $\{H^*_b, H^*_b\} = 0$ when restricted to $\mathcal O_{T[1]X}$.
\end{prop}
\begin{proof} This is an immediate translation of the equality
$d^b d^b = 0$. For readers' convenience, we prove this by direct calculation.
By \eqref{eq:LL}, we  have
$$
\{p_\nu q_\nu^* b, p_\nu q_\nu^* b\} = 0
$$
when restricted to $\mathcal O_{\L^*}$. Finally we note that the vector field
$Q = \{H^*, \cdot\}$ restricts to
$$
\psi^i_* \frac{\del}{\del y^i} + \eta_\alpha^* \frac{\del}{\del q_\alpha^*}
+ \eta_\alpha \frac{\del}{\del q^\alpha}.
$$
Therefore having \eqref{eq:LL} in our mind, we compute
\beastar
\{H^*, p_\nu q_\nu^* b\} & = &\{H^*, p_\nu q_\nu^*(b_\gamma p^\gamma_* + b^i \psi^i)\} \\
& = & \{H^*,H^*\} + \{H^*, p_\nu q_\nu^* b^i y_i^*\} + \{ p_\nu q_\nu^* b^i y_i^*, H^*\}.
\eeastar
But we have $\{H^*,H^*\}  = bH^*$ and compute
$$
\{p_\nu q_\nu^* b, H^*\} = - p_\nu \left(\frac{\del H^*}{\del p_\nu} + q_\nu^* \{b,H^*\}\right).
$$
Both restricts to zero on $\mathcal O_{T[1]X}$ by \eqref{eq:LL} and Lemma \ref{lem:brane}.
This completes the proof of Proposition \ref{prop:H*bH*b}.
\end{proof}

Noting that $\L^* = N^*[1]Y$ is mapped to $\L = T\FF[1]$ under the
isomorphism $\widetilde \omega_X$, this lemma enables us to
restrict the odd vector field $Q$ to $T\FF[1]$. We need to
describe the Lagrangian embedding $T\FF[1] \subset T[1]X$ more
explicitly, and describe the induced directional derivative acting
on
$
\Omega^\bullet(\FF)
$
regarded as a subset of ``functions'' on $T\FF[1]$. (Again we refer
to \cite[Appendix]{oh-park} or \cite{getzler} for the precise explanations of this).

Now we define  $\delta'_b: \Omega^\bullet(\FF) \to \Omega^\bullet(\FF)$ by the formula
\begin{equation}\label{eq:delta'}
\delta'_b(\xi): = \{H^*_b, \widetilde \xi\}_\Omega\Big|_{\L}
\end{equation}
where $\widetilde \xi$ is the extension of $\xi$ in a neighborhood
of $\L \subset T[1]X$: the extension that we use is the lifting of
$\xi \in \Omega^\bullet(\FF)$ to an element of $\Omega^\bullet(U)$
obtained by the (local) Ehresman connection constructed in section
\ref{sec:master}. The condition $Q|_\L\equiv 0$ implies that this
formula is independent of the choice of  (local) Ehresman
connection. We will just denote $\delta'_b(\xi)= \{H^*_b, \xi\}_\Omega$
instead of (\ref{eq:delta'}) as long as there is no danger of
confusion.

Obviously, $\delta'_b$ satisfies $\delta'_b\delta'_b = 0$ because of
$\{H^*_b,H^*_b\} = 0$ by Proposition \ref{prop:H*bH*b}.
Now it remains to verify that this is translated
into the $L_\infty$-relation $\delta^b\delta^ b = 0$ in the tensorial
language which is exactly what we wanted to prove. For this
purpose, we need to describe the map $\delta'_b: \Omega^\bullet(\FF)
\to \Omega^\bullet(\FF)$ more explicitly.

The rest of the argument is precisely the same as the end of
the proof of  \cite[Theorem 9.4]{oh-park}.
By expanding the even function $H^*_b$
above into the power series
$$
H_b^* = \sum_{\ell =1} H_\ell, \quad H_\ell \in \frak l^\ell,
$$
in terms of the degree (i.e., the number of
factors of odd variables $(y_i^*, p_*^\alpha, p_\alpha)$
or the `ghost number' in the physics language)
our definition of $\frak m^b_\ell$ exactly corresponds to the
$\ell$-linear operator
$$
(\xi_1, \xi_2, \cdots, \xi_\ell) \mapsto \{\cdots \{H_\ell,
\xi_1\}_\Omega, \cdots\}_\Omega, \xi_\ell\}_\Omega.
$$
Note that the above power series acting on
$(\xi_1, \cdots, \xi_\ell)$ always reduces to a finite sum and so is
well-defined as an operator. Then by definition, the coderivation
\be\label{eq:deltab}
\delta^b = \sum_{\ell =1}^\infty \widehat{\frak m}^b_\ell
\ee
precisely corresponds to $\delta'_b = \{H^*_b, \cdot\}_\Omega$. Here $\frak m^b_\ell$
is the $b$-deformed $\frak m^b_\ell$-map defined by
$$
\frak m^b_\ell(\xi_1,\xi_2, \cdots, \xi_\ell) = \sum_{k=0}^\infty \frac{1}{k!}\frak m_{\ell+k}
(\underbrace{b,\cdots,b}_{k\, \mbox{times}},\xi_1,\cdots, \xi_\ell).
$$
(See \cite[section 3.6]{fooo} for a general discussion on the deformation of $A_\infty$ structures.
This definition is the symmetrized version thereof.)
The $L_\infty$ relation $\delta^b \delta^b = 0$ then immediately follows
from $\delta'_b \delta'_b = 0$. This finishes the proof.
\end{proof}

For example, under the above translation,
the action of the odd vector field
$Q\mid_\L $ on
$$
\frak l = \bigoplus_{\ell = 0}^{n-k}\frak l^\ell \cong
\bigoplus_{\ell = 0}^{n-k}\Omega^\ell(\FF),
$$
 translates into to the leafwise differential $d_{\FF}^{\bar b}$.
This finishes the proof of Theorem \ref{algebroid}.
\qed
\subsection{Gauge equivalence}
\label{subsec:gauge}

In this section 
we prove
that two strong homotopy Lie algebroids we have associated to two
different splittings are {\it gauge equivalent} or $L^\infty$-isomorphic.
This is the formal analog to the ($C^\infty$)-Hamiltonian equivalence of
the coisotropic submanifolds.  

\begin{defn}\label{Lmorphism}
Let $(C[1], \frak m)$, $(C'[1],\frak m')$ be $L_\infty$-algebras
and $\delta, \, \delta'$ be the associated coderivation. A
sequence $\varphi = \{\varphi_k\}_{k=1}^\infty$ with
$\varphi_k:E_kC[1] \to C'[1]$ is said to be an {\it $L_\infty$-homomorphism}
if the corresponding coalgebra homomorphism $\widehat
\varphi: EC[1] \to EC'[1]$ satisfies
$$
\widehat \varphi \circ \delta
= \delta' \circ \widehat \varphi.
$$
We say that $\varphi$ is an {\it $L_\infty$-isomorphism}, if there
exists a sequence of homomorphisms $\psi=
\{\psi_k\}_{k=1}^\infty$, $\psi: E_kC'[1] \to C'[1]$ such that its
associated coalgebra homomorphism $\widehat \psi: EC'[1] \to
EC[1]$ satisfies
$$
\widehat \psi \circ \widehat \varphi = id_{EC[1]},
\quad
\widehat \varphi \circ \widehat \psi = id_{EC'[1]}.
$$
In this case, we say that two $L_\infty$ algebras,
$(C[1], \frak m)$ and $(C'[1], \frak m')$ are $L_\infty$
isomorphic.
\end{defn}

The following theorem is the $b$-deformed version of  \cite[Theorem 10. 1]{oh-park}.
(See also Definition 8.3.6 \cite{fukaya}.)

\begin{thm}\label{thm:splitting-gauge} The two structures of strong homotopy Lie algebroid
on the null distribution $E=T\FF$ of $(Y,\omega, b)$ induced by two choices of splitting
$\Pi, \, \Pi^\prime$ are canonically $L_\infty$-isomorphic.
\end{thm}
\begin{proof}
We start with the expression of the l.c.s. form $\omega_U$
$$
\omega_U = \pi_Y^*\omega - d\theta_G - \pi_Y^* b \wedge \theta_G
$$
given in (\ref{eq:omega*}) that is canonically constructed on a
neighborhood $U$ of the zero section $E^* = T^*\FF$ when a
splitting $\Pi: TY = G\oplus T\FF$ is provided. To highlight
dependence on the splitting, we denote by $\theta_\Pi$ and
$\omega_\Pi^b$ the one-form $\theta_G$ and the l.c.s. form
$\omega_U$. We will also denote by $\delta_\Pi^b$ the $\delta: EC[1]
\to EC[1]$ corresponding to the splitting $\Pi$.
Here we would like to emphasize that the one-form $\theta_G$
depends only on the splitting $\Pi$ but not depend on the
one-form $b$.

Then for a given splitting $\Pi_0$, we have
\begin{equation}\label{eq:difference}
\omega_\Pi^b - \omega_{\Pi_0}^b = d^{\pi_Y^*b} (\theta_{\Pi_0} - \theta_\Pi).
\end{equation}
In the super language, this is translated into
\begin{equation}\label{eq:super-difference}
H_{\Pi}^b - H_{\Pi_0}^b
= \{H_{\Pi_0}^b, \Gamma\}_{\Omega}
= - \{\Gamma, H_{\Pi_0}^b \}_{\Omega}
\end{equation}
where $\Gamma$ is the function associated to the one-form
$\theta_{\Pi_0} - \theta_\Pi$ which has $deg'(\Gamma) = 0$ (or
equivalently has $deg(\Gamma) = 1$). This function does \emph{not}
depend on $b$. The last identity
comes from the super-commutativity of the bracket and
the fact that $deg(H_{\Pi_0}^b) = 2$ and $deg(\Gamma) =1$.
Once we have established these, the rest of the proof is the same as
that of Theorem 10.1 \cite{oh-park} and so omitted, referring the readers
thereto.
\end{proof}

This theorem then associates a canonical ($L_\infty$-)isomorphism
class of strong homotopy Lie algebras to each l.c.p-s
manifold $(Y,\omega,b)$ and so to each coisotropic submanifold
of l.c.s manifold $(X,\omega_X,\frak b_X)$. As in the symplectic case, it is obvious
from the construction that pre-Hamiltonian diffeomorphisms induce
canonical isomorphism by pull-backs in our strong homotopy Lie
algebroids.

In the point of view of coisotropic embeddings in l.c.s manifolds, this theorem implies
that our strong homotopy Lie algebroids for two
Hamiltonian isotopic coisotropic submanifolds are canonically
isomorphic and so the isomorphism class of the strong homotopy Lie
algebroids is an invariant of coisotropic submanifolds modulo the
Hamiltonian isotopy as in the symplectic case \cite{oh-park}.
(See the relevant discussion in section \ref{sec:bulk} on the general bulk deformations of
coisotropic submanifolds. According to the definitions therein,
Hamiltonian deformations of coisotropic submanifolds correspond to
equivalent bulk deformations.)
This enables us to study the moduli problem of deformations of
l.c.p-s. structures on $Y$ in the similar way as done in \cite{oh-park}.
Up until now, most of our discussions correspond to the l.c.s. analogues of the
deformation theory developed in \cite{oh-park} in the symplectic context.
This effort will finally pay off when we study the moduli problems of
coisotropic submanifolds and its obstruction-deformation theory.
We will be particulary interested in the deformation problems of
Zambon's example in this enlarged categorical setting of
conformally symplectic manifolds.

\section{Moduli problem and the Kuranishi map}
\label{sec:moduli}

In this section, we write down the defining equation
(\ref{eq:masterincoord}) for the graph $\operatorname{Graph}s
\subset TU \subset TE^*$ to be coisotropic in a formal
neighborhood, i.e., in terms of the power series of the section
$s$ with respect to the fiber coordinates in $U$ and study them using (\ref{eq:main}).
Using  the concept of gauge equivalence of the solutions of a  Maurer-Cartan
equation  in \cite[4.3]{fooo} we will study the moduli problem of the
Maurer-Cartan equation (\ref{eq:MC}) of $\frak l^\infty_{(Y,\omega,b)}$.

First we state the following $b$-deformed analogue of
Theorem 11.1 \cite{oh-park} whose proof is the same as that of the latter and
so omitted.

\begin{thm} The equation of the formal power series
solutions $\Gamma \in \frak l^1$
of $(\ref{eq:masterincoord})$
is given by
\begin{equation}\label{eq:MC}
\sum_{\ell = 1}^\infty \frac{1}{\ell !}
\frak m_\ell^b(\Gamma,\cdots, \Gamma) = 0 \quad \mbox{on } \,
\Omega^2(\FF)
\end{equation}
where
\begin{equation}
\Gamma = \sum_{k = 1}^\infty \e^k \Gamma_k \label{eq:Gammapower}
\end{equation}
where $\Gamma_k$'s are sections of $T^*\FF$ and $\e$ is a formal
parameter.
\end{thm}

As in \cite[Remark 11.1]{oh-park} it is possible to interpret (\ref{eq:MC}) as the
condition for the gauge changed weak (or curved) $L_\infty$-structure to
define a (strong) $L_\infty$-structure, i.e., $\frak m^{b,\Gamma}_0 = 0$
or the Maurer-Cartan equation  $\widehat{\frak m}^{b,\Gamma}_1 \circ \widehat{\frak m}^{b,\Gamma}_1 = 0$
for the deformation problem of the corresponding
l.c.p-s. structure $(Y,\omega,b)$. In what follows,  we will use (\ref{eq:main})
to study a formal solution of (\ref{eq:MC}). This description seems to be more
suitable for the study of $C^\infty$ Maurer-Cartan equation which we hope to
pursue in a sequel to the present paper.

By Theorem \ref{thm:coisotropic}, using  (\ref{eq:somegag}),  a solution  $\Gamma$   of (\ref{eq:masterincoord})
is  also  a solution of  (\ref{eq:main}), and therefore
a formal solution  of  (\ref{eq:masterincoord})  is  also  a
formal solution of (\ref{eq:main}). Let us plug  a  formal solution
 $\Gamma$, identified with $p_G ^* \Gamma$, into (\ref{eq:main}),  denoting  $a_i := - d^b (\Gamma_i)$
 and expanding $(\omega- d^b ( \Gamma)) ^{k +1}$.
Noting  that $a_i$ and $\omega ^p$ are   differential forms of even degree,
we abbreviate   wedge products  of them as usual products. As a result we obtain
\begin{equation}
\sum_{N = 0}^\infty \epsilon ^N
 \sum_{\begin{array}{l}
1\le i_1 < \cdots  \cdots  < i_p, \\
1\le s_i,\\
 s_1 + \cdots + s_p \le (k+1)\\
 i_1 s_1 + \cdots  + i_p s_p = N
 \end{array} }
 a_{i_1} ^{s_1} \cdots  a_{i_p} ^{s_p} \binom{k +1} {s_1}\cdots \binom{k+1}{s_p}
 \omega ^{k+1 - s_1 \cdots - s_p}(y) = 0.\nonumber 
\end{equation}

Consequently  for each $N \ge 0$ we have
\begin{equation}
0 = a_{N}(k+1)\omega ^k +\nonumber
\end{equation}

\begin{equation}
+\sum_{\begin{array}{l}
 1\le i_1 < \cdots  \cdots  < i_p\le N-1, \\
 1\le s_i,\\
  s_1 + \cdots  + s_p \le (k+1)\\
 i_1 s_1 + \cdots  + i_p s_p = N
 \end{array}}
 a_{i_1} ^{s_1} \cdots  a_{i_p} ^{s_p} \binom{k +1} {s_1}\cdots \binom{k+1}{s_p} \omega ^{k+1 - s_1 \cdots - s_p}(y).\label{eq:Nsum}
\end{equation}

Note that the number $p$ entering in (\ref{eq:Nsum})  is bounded by  $(k+1)$.

Let us examine  (\ref{eq:Nsum})  for small numbers  $N$.
For $N = 0$   the corresponding term in (\ref{eq:Nsum}) is $\omega ^{k +1}(y) = 0$.

For $N = 1$ the corresponding values in  (\ref{eq:Nsum})  are  $p =1 =i _1 =  s_1$  and the corresponding term is
$$  \omega ^{k} a_1(y) = 0$$
which is  equivalent to
$$  d^{\bar b}_{\FF} \Gamma_1(y)= 0.$$ (See \eqref{eq:sglin}, \eqref{eq:mainlin}.)
So $\Gamma_1$ is a solution of the linearized equation, which is assumed to be given.

For $N = 2$ the corresponding values in (\ref{eq:Nsum})  are $p =1$,
$ i_1 = 1 , s_1 =2$ or $i_1 = 2, s_1 = 1$. The equation (\ref{eq:Nsum}) in this case has the following form
\begin{equation}
 (k+1) \omega^{k} a_2 + \binom{k+1}{2}\omega ^{k-1} a_1 ^2   = 0, \nonumber 
\end{equation}
which is equivalent to
\begin{eqnarray}
- \omega ^k (d ^{\bar b} _\FF \Gamma_2) = {k \over 2} \omega ^{k-1} a_1 ^2, \nonumber \\
\Leftrightarrow - d^{\bar b} _\FF \Gamma_2 = {k\over 2}P_\omega \rfloor (\omega ^{k-1} a_1 ^2), \nonumber\\
\Leftrightarrow  - \omega \wedge d^{\bar b}_\FF \Gamma _2 ={k\over 2}a_1^2, \nonumber\\
\Leftrightarrow  - d^{\bar b} _\FF \Gamma_2 = {1\over 2}P_\omega \rfloor (a_1 ^2), \label {N=2gb}
\end{eqnarray}
where
$P_\omega$ is  the bi-vector in $\Lambda ^2 G$ dual to the restriction of $\omega$ to $G$,
so $\omega (P_\omega) = k$.

We note that  the RHS of (\ref{N=2gb}) is ${1\over 2}\frak m_2 ^b (\Gamma_1, \Gamma_1)$ (cf. (\ref{eq:b-quad})).
Since $\frak m_1^b$ is a derivation of  $\frak m_2^b$  the map
$$ \Omega ^1 (\FF) \times \Omega ^1 (\FF) \to \Omega^2(\FF), \,
(\Gamma_1, \Gamma_2) \mapsto  {1\over 2}P_\omega \rfloor (d^b \Gamma _1\wedge d^b \Gamma_2)$$
induces
the  Kuranishi-Gerstenhaber   bracket
$$KG: H^1_b (Y, \omega) \times H^1_b (Y, \omega) \to H^2_b (Y, \omega), \,
([\Gamma_1], [\Gamma_2] ) \mapsto {1\over 2} [\frak m_2^b (\Gamma_1, \Gamma_2)].$$
Since $\frak m_2^b$ is symmetric,  the Kuranishi-Gerstenhaben  bracket is defined by the Kuranishi map \cite{oh-park}
$$
Kr: H^1_b (Y, \omega) \to H^2_b (Y, \omega), \, [\Gamma_1] \mapsto [\frak m_2^b (\Gamma_1, \Gamma_1)].
$$
\begin{cor}\label{cor:Kr} (cf. \cite[Corollary11.5]{oh-park}).
The moduli problem is formally unobstructed only if $Kr$ vanishes.
\end{cor}

The following Theorem  is a $b$-deformed analog of
\cite[Theorem 11.2]{oh-park} derived in the symplectic case, so we omit  its proof.

\begin{thm}\label{thm:formal} Let $\FF$ be the null foliation of $(Y,\omega,b)$
and $\frak l = \oplus_{\ell =1}^{n-k} \frak l^{\ell}$ be the
associated complex. Suppose that $H^2_b(Y,\omega) = \{0\}$. 
Then for any given class
$\alpha \in H^1_b(Y,\omega)$, (\ref{eq:MC}) has a solution $\Gamma =
\sum_{k=1}^\infty \e^k \Gamma_k$ such that $d_\FF^b(\Gamma_1) =0$
and $[\Gamma_1] = \alpha \in H^1_b(Y,\omega)$. In other words, the
formal moduli problem is unobstructed.
\end{thm}

In general, we say that an element $\Gamma_1\in\ker  d^b_\FF \cap \Omega^1 (\FF)$ is
{\it formally unobstructed},
if there exists a formal  solution to (\ref{eq:MC}) whose  first summand $\Gamma_1$ is the
given one. Similarly, $\Gamma_1\in \ker  d^b_\FF\cap \Omega^1 (\FF) $ is  called
{\it smoothly unobstructed}, if it is  tangent  to a curve of smooth coisotropic deformations.

Note that  Hamiltonian  diffeomeorphisms on $U$ that are close to the identity  act on the space of  formal deformations
by acting on each summand  $\Gamma_l$ in (\ref{eq:Gammapower}). (If a diffeomorphism $\phi: U \to U$ is close to the identity and $\Gamma: Y \to U\subset G^\circ $ is a section, then
the composition $\pi\circ \phi \circ \Gamma: Y \to Y$ is a diffeomorphism, hence there exists a diffeomorphism $g\in Diff (Y)$ such that
$\phi \circ \Gamma \circ g: Y \to U \subset G^\circ $ is a section.)  They also act on
the  space of smooth  coisotropic deformations
by an obvious way.
The following Lemma is straightforward, so we omit its proof.

\begin{lem}\label{lem:ham}  Given a function $f \in C^\infty (Y)$
and an element $a \in \Omega ^1 (\FF)\cap \ker d^b_\FF$ the following   assertions hold:
\begin{enumerate}
\item $a$ is formally unobstructed, if and only if   $a + d^{\bar b} _\FF  f$ is  formally
unobstructed.
\item $a$ is  tangent to  a curve  $\gamma(t)$ of coisotropic deformations  of $Y$ in $(U, \omega_G, \pi ^b)$
then $a + d^{\bar b}_\FF f$ is tangent to the curve $\phi_t \circ \gamma_t$, where $\phi_t$'s are Hamiltonian
diffeomorphisms  close  to the identity on  $(U, \omega _G, \pi^* b)$  and generated by an extension of  Hamiltonian $f$  to $U$.
\end{enumerate}
\end{lem}

Now  we study the moduli of the solution of the Maurer-Cartan equation
under the action induced by   Hamiltonian  diffeomorphisms on  a neighborhood
$(U, \omega_U, \pi^*b)$ of $Y$, taking into account Theorem \ref{extlcps}.
Here
we follow  the ideology in \cite[4.3]{fooo}. First we need introduce the notion of
a model of the product of  $[0,1]$ with  an $L_\infty$-algebra  $(C[1], \frak m)$,  which
is also a $L_\infty$-algebra, imitating  the analogous notion for  $A_\infty$-algebras,
introduced in \cite[4.2]{fooo}.

\begin{defn}\label{def:model} (cf. \cite[Definition 4.2.1]{fooo})
An $L_\infty$-algebra $(\bar C[1], \bar{\frak m})$  together
with  $L_\infty$-homomorphisms
$$Incl: EC[1] \to \bar C[1], \: Eval_{s=0} : E\bar C[1] \to C, \:  Eval_{s=1}: E\bar C \to C$$
is said to be {\it a model of  $[0,1] \times (C[1], \frak{m})$}, if the following  holds:
\begin{enumerate}
\item $Incl:E_kC[1] \to \bar C[1]$ is zero unless $k = 1$ The same holds for $Eval_{s=0}$ and $Eval_{s=1}$.
\item  $Eval_{s=0}\circ Incl = Eval_{s=1}\circ Incl = $ identity.
\item $Incl_1: (C[1], \frak m_1)  \to (\bar C[1], \bar{\frak m})$ is a cochain homotopy equivalence and
$(Eval_{s=0})_1, \, (Eval_{s=1})_1: (\bar C[1], \bar{\frak m}_1) \to (C[1], \frak m_1)$
are cochain homotopy equivalences.
\item The (cochain)  homomorphism
$(Eval_{s=0})_1 \oplus (Eval_{s=1})_1 : \bar C[1] \to C[1] \oplus C[1]$ is surjective.
\end{enumerate}
\end{defn}

\begin{defn}[Maurer-Cartan moduli space] (cf. \cite[Definition 4.3.1]{fooo})  We say that two solutions
$\Gamma_1$, $\Gamma_2$ of \eqref{eq:MC} are gauge-equivalent if  there exist a model
$(\bar C [1], \bar {\frak m})$    of $[0,1] \times (C[1], \frak{m})$
 and a  solution  $\tilde \Gamma$ of the Maurer-Cartan equation of  $(\bar C [1], \bar {\frak m})$  such that
$ Eval_{s=0 *}(\tilde \Gamma) = \Gamma_1, \, Eval_{s=1 *}(\tilde \Gamma) = \Gamma_2$.
 We say such a $\tilde \Gamma$ {\it  a homotopy}  from $\Gamma_1$ to $\Gamma_2$.
We denote by $\mathfrak M_{formal}(Y,\omega,b)$ the set of gauge equivalence classes of
Maurer-Cartan solutions.
\end{defn}

Note that any Hamiltonian diffeomorphisms $\phi_t$, $\phi_0 = Id$, on $(U, \omega_U,  d ^{\pi^*b})$ which are close to the identity,  provides a homotopy between a formal solution $\Gamma$ to (\ref{eq:MC}).
Denote by ${\mathfrak M}_{Ham}(Y,\omega,b)$ the set of Hamiltonian equivalent classes  of formal coisotropic  deformations of $Y$ in $U$.  Using the argument of the proof  Theorem \ref{thm:splitting-gauge}
we obtain easily a natural map $\mathfrak M_{Ham}(Y,\omega,b)\to \mathfrak M_{formal}(Y,\omega,b)$.

\begin{rem}\label{rem:hamequi}
Theorem \ref{thm:coisotropic} implies that    there is a one-one correspondence between  of  coisotropic deformations  of   a coisotropic submanifold $Y \subset  (U, \omega_U, \pi^* b)$
and   the set of deformations  $(\omega', b)$ of the l.c.p-s. form  $(\omega, b)$   that are of the same rank as $\omega$ and $\omega '- \omega = d^b s$  for some $s\in \Omega ^1 (\FF)$.
The Hamiltonian  equivalence of the coisotropic deformations induces a ``Hamiltonian" equivalence   on  this set  of l.c.p-s. forms on $(Y, \omega, b)$.
\end{rem}

A deeper analysis on the relationship between the equations (\ref{eq:MC}) and (\ref{eq:Nsum}),
using some ideas in \cite{LV}, will be given in a sequel to the present paper.
In particular, it was raised as a question in \cite{oh-park} whether the  $C^\infty$-analog to
Theorem \ref{thm:formal} holds or not. We hope to study and answer to this question
 in the sequel.

\begin{rem} Among coisotropic deformations of $Y$  there are  special
deformations  respecting the  leaf $\FF$, i.e. those  deformations $\Gamma$
whose  associated null foliation
$\FF$   stay unchanged, or equivalently, $\FF \subset \ker d^b \Gamma$.
For instance, if $Y$ is Lagrangian  all  coisotropic
deformations respect $\FF= Y$. These  deformations form a linear space,
therefore,  they are  smoothly unobstructed. Clearly, they are invariant  under
infinitesimally Hamiltonian  actions.
A particular case  has been considered by Ruan \cite{Ruan}.
\end{rem}

\section{Deformations of l.c.s. structures on $X$}
\label{sec:bulk}

In this section we derive     formulas (\ref{eq:def1a}), (\ref{eq:def2a}) describing the  Zariski tangent  space of  the set $\frak M_{lcs}(X)$ of equivalent classes
of l.c.s. structures on  a manifold $X$.


\begin{defn} We call a smooth one-parameter family $(X,\omega_t,\frak b_t)$ of
l.c.s structures for $-\e \leq t \leq \e$ a \emph{bulk-deformation}.
\end{defn}

Since nondegeneracy is an open condition, we can represent a deformation $\omega_t$ with
$$
\frac{\del \omega_t}{\del t}\Big|_{t=0} = \kappa.
$$
The l.c.s. condition can be written as
\be\label{eq:lcscondition}
\begin{cases}
d\omega_t +  \frak b_t \wedge \omega_t = 0,\\
d\frak b_t = 0.
\end{cases}
\ee
In fact, since we assume $\dim X \geq 4$, $\omega_t$ uniquely determines $\frak b_t$.
So we will focus on the deformation of $\omega_t$.
By differentiating \eqref{eq:lcscondition} with respect to $t$ at $0$, we
obtain
\be\label{eq:att=0}
d\kappa + \frak b_0 \wedge \kappa +\frac{\del \frak b_t}{\del t}\Big|_{t=0}\wedge  \omega_X=0.
\ee
Therefore we immediately derive the following description of Zariski tangent space of
the set of l.c.s. structures.

\begin{lem} Let $(X,\omega_t,\frak b_t)$ be a bulk-deformation of l.c.s. structure on $X$
with $(\omega_0,\frak b_0) = (\omega_X, \frak b)$. Denote
$$
\frac{\del \omega_t}{\del t}\Big|_{t=0} = \kappa,  \quad \frac{\del \frak b_t}{\del t}\Big|_{t=0} = \frak c
$$
Then $(\kappa,\frak c)$ satisfies
\be\label{eq:bulk-MC-derivative}
d^{\frak b}\kappa = - \frak c \wedge \omega_X, \quad d\frak c = 0.
\ee
\end{lem}

Since  two  l.c.s. forms $ e^{f_t} \omega_t$ and $\omega_t$ are equivalent  for  $f_t\in C^\infty (X)$, two   infinitesimal deformations $(\kappa, \frak c)$ and $(\kappa',\frak c')$ of
$(X, \omega_X, \frak b)$  are equivalent if  there is a  function $f \in C^\infty (X)$ such that
\begin{equation}
\kappa = -f\omega_X +\kappa' \text{ and } \frak c = \frak c'+ df.\label{eq:equi1}
\end{equation}


\begin{defn} We call a pair $(\kappa,\frak c)$ an infinitesimal deformation of
$(X,\omega_X,\frak b)$ when it satisfies
\eqref{eq:bulk-MC-derivative} or equivalently
$$
d^{\frak b} \kappa = - \frak c \wedge \omega_X, \quad d\frak c = 0.
$$
\end{defn}

Now we recall the following from Definition \ref{morphism}

\begin{defn} We say $(X,\omega,\frak b)$ is diffeomorphic  to $(X',\omega', \frak b')$
if there exists a l.c.s. diffeomorphism  $\phi: X \to X'$.
We denote by $\frak{LCS}(X)$ the set of l.c.s. structures on $X$
and $\frak M_{lcs}(X)$   the set of equivalence  classes of l.c.s. structures on $X$.
\end{defn}

The following is the infinitesimal analog to this definition, taking into account (\ref{eq:equi1}).

\begin{defn}\label{def:equi1}
We say two infinitesimal deformations $(\kappa',\frak c')$, $(\kappa, \frak c)$
of $(X,\omega_X, \frak b)$ are {\it equivalent}, if there exist a  vector field $\xi$ of $X$  and a function $f\in C^\infty (X)$ such that
$$
\kappa' = -f\cdot \omega_X + \kappa + \CL_\xi \omega_X, \quad \frak c' =  \frak c + \CL_\xi \frak b + df.
$$
We denote by $\operatorname{Def}(X,\omega_X,\frak b)$  the set of equivalence  classes of
infinitesimal deformations of $(X,\omega_X,\frak b)$.
\end{defn}

By definition, $\operatorname{Def}(X,\omega_X,\frak b)$  is the Zariski (or formal) tangent space of
$\frak{M}_{lcs}(X)$  at $(\omega_X,\frak b)$.



Next, we provide an explicit description of  the Zariski tangent space  $\operatorname{Def}(X,\omega_X,\frak b)$.


\begin{defn} Define a map $S(\omega_X, \frak b): Vect (X) \oplus \R \to   (\ker d^{\frak b} \cap \Omega^2 (X))$  by
$$
S(\omega_X, \frak b) (\xi, c): = d^{\frak b} (\xi \rfloor \omega_X) - c\, \omega_X.
$$
\end{defn}

We divide our description of $\operatorname{Def}(X,\omega_X,\frak b)$ into two different cases depending on the cohomological property of $[\omega_X] \in H^2_{\frak b} (X)$.

We start with the case where the linear map $L : H^1 (X, \R) \to  H^3_{\frak b} (X, \R),\, [\alpha]
\mapsto [\alpha] \wedge [\omega_X],$ is injective. In this case,
any solution of  (\ref{eq:bulk-MC-derivative}) 
is of the form
$$
(\kappa = - f\cdot \omega_X + \beta, \,  \frak c = df),
$$
where
$$
f\in C^\infty (X) \text{ and } \beta  \in \ker d^{\frak b} \cap \Omega ^2 (X).
$$
By Definition \ref{def:equi1}
$( - f\cdot \omega_X + \beta, df)$ is  equivalent to $(\beta, 0)$. The Cartan formula  yields
\beastar
\LL_\xi \frak b  & = & d(\frak b(\xi)) \\
\LL_\xi  (\omega_X) & = &\xi \rfloor  d\omega_X + d(\xi \rfloor \omega_X) =
 -\frak b(\xi) \omega_X + \frak b \wedge (\xi \rfloor \omega_X)  +  d (\xi \rfloor \omega_X).
\eeastar
Hence $(\beta, 0)$ is infinitesimally equivalent to zero
if and only if  there exist  a function $g \in C^\infty(X)$   and a vector field $\xi$ on  $X$  such that
\begin{eqnarray}
\beta =  d ^{\frak b} (\xi \rfloor \omega_X) -( g+ \frak b (\xi))\omega_X  \text{ and } 0 =  d(g+(\frak b(\xi)).\nonumber \\
\Leftrightarrow \beta =  d ^{\frak b} (\xi \rfloor \omega_X) -c\omega_X \text{ and } g+\frak b(\xi) = c.\nonumber
\end{eqnarray}
  Therefore  we have
\begin{equation}
\operatorname{Def}(X,\omega_X, \frak b) = ((\ker  d^b \cap \Omega ^2 (X))/ S(\omega, \frak b) (Vect(X)\oplus \R) =$$
$$= H^2_\frak b (X)/\langle \omega_X\rangle _{\otimes \R}.\label{eq:def1a}
\end{equation}
In particular, if $\frak b = 0$, i.e. $(X,\omega_X,\frak b)$ is actually a symplectic manifold,   then
\begin{equation}
\operatorname{Def}(X,\omega_X) = H ^2 (X,\R)/\langle \omega_X\rangle _{\otimes \R} .\label{eq:def1as}
\end{equation}

Next, we consider the case where the  linear map $L : H^1 (X, \R) \to  H^3_{\frak b} (X, \R),\, [\alpha] \mapsto [\alpha]
\wedge [\omega_X],$ is not injective. In this case
any solution of (\ref{eq:bulk-MC-derivative}) 
is of form
$$
(\kappa = - f\cdot \omega_X + \beta + \theta, \,  \frak c = df + \gamma),
$$
where
$$
f\in C^\infty (X), [\gamma] \not = 0  \in H^1 (X, \R), \, \gamma\wedge \omega_X = d^{\frak b}\theta,
\text{ and } \beta  \in \ker d^{\frak b} \cap \Omega ^2 (X).
$$
Again
the argument above implies that $(\kappa= - f\cdot \omega_X + \beta +\theta, \, \frak c = df +\gamma)$
is  infinitesimally equivalent to zero, if
and only if there exist  a function $g\in C^\infty (X)$   and a vector field $\xi$ on  $X$  such that
$$\gamma =- d(\frak b  (\xi) + g) \text{ and } \beta + \theta  = d^{\frak b} (\xi \rfloor \omega_X) - (g+ \frak b (\xi))\omega_X.$$
It follows  that $[\gamma] = 0 \in H^1 (X)$. Hence in this case  we have
\begin{equation}
\operatorname{Def}(X,\omega_X, \frak b)  =\ker L \oplus (\ker  d^{\frak b} \cap \Omega ^2 (X) /S(\omega_X, \frak b) (Vect(X) \oplus \R) =$$
$$= \ker L \oplus  H^2_\frak b (X)/\langle \omega_X\rangle _{\otimes \R}.\label{eq:def2a}
\end{equation}
In particular, if $\frak b = 0$, i.e. $(X,\omega_X,\frak b)$ is actually a symplectic manifold, then
\begin{equation}
\operatorname{Def}(X,\omega_X) = \ker L \oplus H^2 (X, \R) /\langle \omega_X \rangle_{\otimes \R}  .\label{eq:def2as}
\end{equation}

\begin{rem}\label{rem:banyaga} In \cite[Theorem 2]{banyaga:lcs}  Banyaga  considered   deformations of l.c.s. forms with a given Lee one-form.
\end{rem}

\section{Bulk deformations of coisotropic submanifolds; Zambon's example re-visited}

In this section we consider  bulk deformations of l.c.s. forms on a  l.c.s. manifold $X$ under which a given compact coisostropic submanifold  $Y$ stays  coisotropic.
Then we study  bulk coisotropic deformations of $Y$ under such  bulk deformations of l.c.s. forms (Definition \ref{def:bulkcois}, Lemma \ref{bpre}, Lemma \ref{lem:bulkcois}, Theorem \ref{thm:bulkisod}).  Finally we re-examine the Zambon example  under   bulk coisotropic deformations and show that it is still obstructed (Theorem \ref{thm:zambonlcs}.)

Given a coisotropic  submanifold $i: Y  \to (X, \omega_X, \frak b)$ we say that
a bulk deformation $(\omega_t, \frak b_t)$  respects $Y$, if $Y$  remains  coisotropic
in $(X, \omega _t, \frak b_t)$.
By the  normal  form theorem \ref{thm:normalform}, if $Y$ is compact, there exist
a neighborhood $U$ of $Y$ in $X$,  a family  of diffeomorphisms $\phi_t : U \to U$
and  a family of  smooth function $f_t \in C^\infty (U)$ such that  for all
  $t \in [-\varepsilon, \varepsilon]$ we have
 $$\phi_t (Y) = Id, $$
 $$\phi_t^*  (\pi^*  i^* \omega _t - d^{ \pi ^* i ^*  b_t}  \theta _G) = e ^{ f_t} \omega _t.$$
    Here  we identify $U$ with a neighborhood of the zero section of $E^* = E^*_t$  as in section \ref{sec:coiso-normal}.

\begin{defn}\label{def:bulkcois}  Assume that $Y$ is  a  coisotropic  submanifold of
$(U, \omega_U, d^{\pi ^* b}\theta_G )$.
A deformation $\Gamma_t: Y \to U$  is  called {\it a bulk coisotropic  deformation},
if there exists  a family  of   l.c.p-s. form $(\bar \omega _t, b_t)$  of constant rank on $Y$  with
$\bar \omega _0 =  i^* \omega_U$, $b_0  = b$  and for each $t\in [-\varepsilon, \varepsilon]$
(the graph of) $\Gamma_t$   is coisotropic      in $(U, \pi^* \bar \omega_t,  d^{\pi ^* b_t}\theta_G)$.
An infinitesimal  coisotropic deformation $\Gamma_1  \in \ker d^{\bar b}_\FF \cap \Omega ^1 (\FF)$
is called  {\it  formal bulk unobstructed}, if  there exist a  formal bulk deformation of
$(\bar \omega_0, b_0)$ and  a formal bulk coisotropic  deformation $\Gamma$ whose first term  is the given
 $\Gamma_1$.   An infinitesimal coisotropic deformation
 $\Gamma_1  \in \ker d^{\bar b}_\FF \cap \Omega ^1 (\FF)$  is called  {\it  smoothly  bulk unobstructed},
 if  there exists
 a smooth bulk coisotropic  deformation $\Gamma_t$ such that $(d/dt)_{t = 0} \Gamma_ t  = \Gamma_1$.
 \end{defn}

 The following Lemma is a direct consequence of Theorem \ref{thm:coisotropic}.

 \begin{lem}\label{lem:bulkcois}   A deformation $\Gamma_t$ is a bulk  coisotropic deformation,  if and only if there is a bulk deformation $(\bar \omega_t, b_t)$ of  the l.c.p-s. form $(\bar \omega_0, b_0)$  on $Y$ such that
 \begin{equation}
  (\bar \omega_t)^{k+1} = 0, \label{bulk1}
 \end{equation}
 \begin{equation}
  (\bar \omega_t - d ^{b _t} \Gamma_t) ^{ k+1} = 0.\label{eq:var}
  \end{equation}
 \end{lem}

  The following Lemma is obtained   straightforward.

  \begin{lem}\label{bpre} Let  $(\omega_t, b_t)$ be
a smooth family of l.c.p-s. structures of constant rank $2k$ on   $(Y, \omega_0)$ and denote
$$
\frac{\del \omega_t}{\del t}\Big|_{t=0} = \kappa, \quad \frac{\del  b_t}{\del t}\Big|_{t=0} =  c.
$$
Then $(\kappa, c)$ satisfies
\be\label{eq:dkappap}
d^{ b}\kappa = -  c \wedge \omega_Y, \quad d  c = 0
\ee
\be \label{eq:rank}
\omega_0 ^k \wedge \kappa = 0 \Leftrightarrow \kappa|_\FF = 0.
\ee
Here $\FF$ is the null foliation of $(Y, \omega_0)$.  Furthermore two equivalent
deformations generates equivalent  $(\kappa,\frak b)$.
 \end{lem}

 The discussion in the previous section   can be repeated word-for-word  for  bulk-deformations
 of an l.c.p-s. form $(\bar \omega, b)$ on $Y$, except that  we
need to take care of $\kappa$, $\beta$ (resp. $\beta +\theta$)
so that  their restriction to $\FF$ vanishes.  Equivalently,
they are in the differential ideal $\CI (\FF)$ generated by $T\FF^\circ$.
We define a subset
$$
\Omega^i(Y,\omega, \FF): = \Omega^i (Y) \cap \CI (\FF)
$$
which defines a differential submodule of $\Omega^i(Y)$ with respect to $d^b$.
Denote its cohomology by
$$
H ^i_b (Y, \omega, \FF) : = \frac{\ker d^b \cap \Omega^i (Y, \omega,\FF)}{d^b (\Omega^{i-1} (Y, \omega, \FF))}.
$$
Note that the wedge product  with $\omega$ restricts to a map  $\CI (\FF) \to \CI(\FF)$.
The map  descends to a  map $L: H^1_0(Y, \omega, \FF) \to H^3_b (Y, \omega, \FF)$.

The following theorem is  obtained  using the same arguments  in the previous section, so we omit its proof.

\begin{thm}\label{thm:bulkisod} The  space $Def(Y, \omega, b)$ of infinitesimal  equivalent bulk-deformations of
l.c.p-s. form $(\omega, b)$ on  $Y$
is isomorphic  to the space  $H^2_b (Y, \omega, \FF)/\langle \omega \rangle _{\otimes \R} \oplus \ker L.$
\end{thm}
Now we are ready to analyze Zambon's example.

\begin{exm}
We recall Zambon's example from \cite{zambon}, \cite{oh-park}.
Let $(Y,\omega)$ be the
standard 4-torus $T^4 = \R^4/\Z^4$ with coordinates $(y^1, y^2,
q^1, q^2)$ with the pre symplectic form
$$
\omega_Y = \bar \omega_0=  dy^1 \wedge dy^2, \,  b_0 = 0.
$$
Note that the null foliation is provided by the 2-tori
$$
\{y^1 = const,\, y^2 = const \},
$$
and it also carries the transverse foliation given by
$$
\{q^1 = const, \, q^2 = const \}.
$$
The canonical symplectic thickening is given by
\begin{eqnarray*}
E^* & = & T^4 \times \R^2 = T^2 \times T^*(T^2), \\
\omega & = & dy^1\wedge dy^2 + (dq^1 \wedge dp^1 + dq^2 \wedge dp^2),
\end{eqnarray*}
where $p^1, \, p^2$ are the canonical conjugate coordinates of
$q^1, \, q^2$.


It follows that the transverse curvature $F \equiv 0$ and so
all $\frak m_\ell = 0$ for $\ell \geq 3$ and
the Maurer-Cartan equation (\ref{eq:MC}) becomes the quadratic equation (cf. (\ref{N=2gb}))
\begin{equation}\label{eq:tori}
- d_\FF(\Gamma_2)  = {1\over 2} P_{\omega_Y} \rfloor (d\Gamma_1) ^2.
\end{equation}



In \cite{zambon}, \cite{oh-park}, the one-form
$$
\Gamma_1 = \sin (2\pi y^1) dq^1 + \sin (2 \pi y^2) dq^2,
$$
was shown to be obstructed by showing  that $Kr ([\Gamma_1]) \not  = 0$. This can be also  shown by
computing the RHS of (\ref{eq:tori})
$$
{1\over 2} P_{\omega_Y} \rfloor (d\Gamma_1 ) ^2  = (\partial y_1 \wedge \partial y_2) \rfloor
4\pi^2 \cos (2\pi y^1) \cos (2\pi y^2) dy^1 dq^1 dy^2 dq ^2 =$$
$$=- 4\pi^2 \cos (2\pi y^1) \cos (2\pi y^2) dq^1 dq ^2,$$
which cannot be  a differential  $d _\FF(-\Gamma_2)$, since the integration
of it over   a generic leaf $T^2(y^1, y^2)$ of $\FF$ is not zero.


\begin{thm}\label{thm:zambonlcs} $\Gamma_1 $ is  formally  bulk obstructed.
\end{thm}

\begin{proof}  Assume the opposite, i.e. there is a formal  bulk deformation $\omega_t = \sum_{i=0}^\infty  t^i \bar \omega _i$   and
 bulk deformation $\Gamma_t = \sum_{ i =1} ^\infty t^i\Gamma_i$.   The equation (\ref{bulk1}) implies
\begin{equation}
\sum_{ i = 0} ^l \bar \omega_i \bar \omega _{l -i } = 0 \text{ for all } 0\le l \le \infty .\label{bulk2}
\end{equation}
 Furthermore $d^{b_t} \omega _t = 0$ is equivalent to the following:
 \begin{equation}
  db_i = 0 \text{ for  all } 0\le i \le \infty \label{eq:bi}
 \end{equation}
\begin{equation}
\text{ and }  0 =  d^{b_0} \bar\omega_i  +\sum_{1\le k \le i} b_k\wedge \bar\omega_{ i-k}  \text{ for  all } 1\le i \le \infty. \label{eq:omi}
\end{equation}

Next,  from (\ref{bulk2}), for $l = 1$, we obtain
$$\bar \omega_0 \wedge \bar \omega _1 = 0 $$
\begin{equation}
\Leftrightarrow \bar \omega_1 = dy^1 \wedge \alpha ^1 + dy ^2 \wedge \alpha ^2, \label{eq:z1}
\end{equation}
where $\alpha ^i \in \Omega ^1 (Y)$.
Further, we obtain from (\ref{eq:omi})  for $i = 1$
\begin{equation}
d\bar \omega _1 = - b_1 \wedge  dy^1\wedge dy^2.\label{eq:z2}
\end{equation}
From (\ref{eq:z2}) and (\ref{eq:bi}) it follows that $[b_1] \in \ker L: H^1 (Y) \to H^3 (Y), \, \gamma\mapsto  \gamma \wedge [\omega_Y]$. Hence we obtain
\begin{equation}
b_1 = b_1 ^1 dy^1 + b^2_1 dy^2 + df \text{ where } b_1 ^i \in \R \text{ and } f \in C^\infty (Y).\label{eq:z3}
\end{equation}
Now from (\ref{bulk2}) for $l = 2$ we obtain
\begin{equation}
\bar \omega_0 \wedge \bar \omega_2  + \bar \omega_1^2  = 0.\label{eq:om2}
\end{equation}
Next, using (\ref{eq:om2}) and (\ref{eq:z3}), we derive from (\ref{eq:var}), looking at the   coefficient of $t^2$,
\begin{eqnarray}
\omega_0 \wedge(d\Gamma_2 + b_1 \wedge \Gamma_1)  + (d\Gamma_1) ^2 = 0\nonumber\\
\Leftrightarrow d_\FF \Gamma_2 + d_\FF(f\cdot \Gamma_1) = 4\pi^2 \cos 2\pi y^1 \cos 2\pi y^2 dq ^1 d q ^2.\label{eq:g2}
\end{eqnarray}
As we have computed, the RHS  of (\ref{eq:g2}) is not equal to zero in $H^2_b(Y, \omega_Y)$.  So the  equation (\ref{eq:g2}) for $\Gamma_2$  and $f$  does not have a solution. This completes the proof of Theorem \ref{thm:zambonlcs}.
\end{proof}

\end{exm}

\section{Appendix}

\subsection{Leaf space connection and curvature}
\label{subsec:splittings}

In this subsection, we recall some basic definitions and
properties of the leaf space connection borrowing the exposition
of \cite[section 3]{oh-park}. We refer readers thereto for the proofs
of all the results stated without proof in the present subsection.

Let $\FF$ be an arbitrary foliation on a smooth manifold $Y$.
Following the standard notations in the
foliation theory, we define the normal bundle $N\FF$ and conormal
bundle $N^*\FF$ of the foliation $\FF$ by
$$
N_y\FF:= T_yY/E_y, \quad N^*_y\FF: = (T_y/E_y)^* \cong E_y^\circ
\subset T_y^*Y.
$$
In this vein, we will denote $E=T\FF$ and $E^*=T^*\FF$
respectively, whenever it makes our discussion more transparent.
We have the natural exact sequences
\begin{eqnarray}\label{eq:exactseq}
0 & \to &  T\FF \to TY \to N\FF \to 0, \\
0 &\leftarrow & T^*\FF \leftarrow T^*Y \leftarrow N^*\FF \leftarrow 0.
\end{eqnarray}
The choice of splitting $TY = G\oplus T\FF$ may be regarded as a
``connection'' of the ``$E$-bundle'' $TY \to Y /\sim$ where $Y/\sim$
is the space of leaves of the foliation on $Y$. Note that
$Y/\sim$ is {\it not} Hausdorff in general. We will indeed call
a choice of
splitting {\it a leaf space connection of $\FF$} in general.

We can also describe the splitting in a more invariant way as
follows: Consider bundle maps $\Pi: TY \to TY$ that satisfy
$$
\Pi^2_x = \Pi_x, \, \operatorname{im }\Pi_x = T_x\FF
$$
at every point of $Y$, and denote the set of such projections by
$$
\AA_E(TY) \subset \Gamma(Hom(TY,TY)) = \Omega^1_1(Y).
$$
There is a one-one correspondence between the choice of splittings
(\ref{eq:splitting}) and the set $\AA_E(TY)$ provided by the correspondence
$$
\Pi \leftrightarrow G:= \ker \Pi.
$$
If necessary, we will denote by $\Pi_G$ the element with $\ker \Pi
=G$ and by $G_\Pi$ the complement to $E$ determined by $\Pi$. We
will use either of the two descriptions, whichever is more
convenient depending on the circumstances.

Next we recall the notion of curvature of
the $\Pi$-connection.

\begin{defn}\label{curvature}
Let $\Pi \in \AA_E(TY)$ and denote by $\Pi: TY = G_\Pi \oplus
T\FF$ the corresponding splitting. The {\it transverse
$\Pi$-curvature} of the foliation $\FF$ is a $T\FF$-valued
two form defined on $N\FF$ as follows: Let $\pi: TY \to N\FF$ be
the canonical projection and
$$
\pi_\Pi: G_\Pi \to N\FF
$$
be the induced isomorphism. Then we define
$$
F_\Pi: \Gamma(N\FF) \otimes \Gamma(N\FF) \to \Gamma(T\FF)
$$
by
\begin{equation}\label{eq:Fdef}
F_\Pi(\eta_1, \eta_2): = \Pi([X,Y])
\end{equation}
where $X = \pi_\Pi^{-1}(\eta_1)$ and $Y= \pi_\Pi^{-1}(\eta_2)$
and  $[X,Y]$ is the Lie bracket on $Y$.
\end{defn}

The following proposition justifies the name {\it
transverse $\Pi$-curvature} which plays a crucial role
in our description of the strong homotopy Lie algebroid associated
to the pre-symplectic manifold $(Y,\omega_Y)$ (and so of
coisotropic submanifolds) and its Maurer-Cartan equation.
We refer to \cite{oh-park} for its proof.

\begin{prop}\label{welldefined} Let $F_\Pi$ be as above.
For any smooth functions $f, \, g$ on $Y$
and sections $\eta_1, \, \eta_2$ of $N\FF$, we have the identity
$$
F_\Pi(f\eta_1,g\eta_2) = fgF_\Pi(\eta_1,\eta_2)
$$
i.e., the map $F_\Pi$ defines a well-defined section as an element in
$\Gamma(\Lambda^2(N^*\FF)\otimes T\FF)$.
\end{prop}

In the foliation coordinates $(y^1,\cdots,y^{\ell},q^1, \cdots,
q^{m-\ell})$, $F_\Pi$ has the expression
\begin{equation}
F_\Pi = F^\beta_{ij} \frac{\del}{\del q^\beta} \otimes dy^i \wedge
dy^j \in \Gamma(\Lambda^2(N^*\CF) \otimes T\FF),
\end{equation}
where
\be\label{eq:Fbetaij}
F^\beta_{ij}= \frac{\del R^\beta_j}{\del y^i} - \frac{\del
R^\beta_i}{\del y^j} + R^\gamma_i \frac{\del R^\beta_j}{\del
q^\gamma} -R^\gamma_j \frac{\del R^\beta_i}{\del q^\gamma}.
\ee

We next recall the relationship between $F_{\Pi_0}$ and $F_{\Pi}$.
Note that with respect to the given splitting
$$
\Pi_0:\, TY = G_0\oplus T\FF \cong N\FF \oplus T\FF
$$
any other projection $\Pi: TY \to TY$ can be written as the
following block matrix
$$
\Pi = \Big(\begin{matrix}
0 & 0 \\
B & Id
\end{matrix}\Big)
$$
where $B = B_{\Pi_0\Pi}\circ \pi_{G_0}: G_0 \to T\FF$ is the
bundle map  which is uniquely
determined by $\Pi_0$ and $\Pi$ and vice versa. The following
lemma shows their relationship in coordinates.

\begin{lem}\label{relation}
Let $F_{\Pi}$ and $F_{\Pi_0}$ be the transverse $\Pi$-curvatures
with respect to $\Pi$ and $\Pi_0$ respectively, and let
$B=B_{\Pi_0\Pi}$ be the bundle map mentioned above.
In terms of the foliation coordinates, we have
\begin{eqnarray}
F^\beta_{ij} & = F^\beta_{0,ij} + \Big(\frac{\del B_j^\beta}{\del y^i}
-\frac{\del B^\beta_i}{\del y^j} +
R^\alpha_i\frac{\del B^\beta_j} {\del q^\alpha} -
R^\alpha_j\frac{\del B^\beta_i} {\del q^\alpha} +
B^\alpha_i\frac{\del R^\beta_j} {\del q^\alpha} -
B^\alpha_j\frac{\del R^\beta_i} {\del q^\alpha}
\Big) \nonumber \\
& \quad + \Big(B_i^\alpha \frac{\del B^\beta_j} {\del q^\alpha}
-  B^\alpha_j\frac{\del B^\beta_i} {\del q^\alpha}\Big)
\label{eq:coord-relation}
\end{eqnarray}
\end{lem}

Now we provide an invariant description of the above formula
(\ref{eq:coord-relation}). Consider the sheaf
$\Lambda^\bullet(N^*\FF)\otimes T\FF$ and denote by
$$
\Omega^\bullet(N^*\FF; T\FF): =
\Gamma(\Lambda^\bullet(N^*\FF)\otimes T\FF)
$$
the group of (local) sections thereof. For an invariant
interpretation of the above basis of $G_x$ and the transformation
law (\ref{eq:coord-relation}), we need to use the notion of {\it
basic vector fields} (or {\it projectable vector fields})
which is standard in the foliation
theory (see e.g., \cite{MM}) :  Consider the Lie subalgebra
$$
L(Y,\FF) = \{ \xi \in \Gamma(TY) \mid ad_\xi(\Gamma(T\FF)) \subset
\Gamma(T\FF) \}
$$
and its quotient Lie algebra
$$
\ell(Y,\FF) = L(Y,\FF) / \Gamma(T\FF).
$$
An element from $\ell(Y,\FF)$ is called a {\it transverse vector
field} of $\FF$. In general,
there may not be a global basic lifting $Y$ of a given transverse vector
field. But the following lemma shows that  this is always possible
locally.

\begin{lem}\label{PiYx=v} Let $x_0 \in Y$ and
$v \in N_{x_0}\FF$. Then there exists a local basic vector
field $\xi$ in a neighborhood of $x_0$ such that it is tangent to $G$
$$
\pi(\xi(x_0)) = v
$$
where $\pi: TY \to N\FF$ is the canonical projection.
\end{lem}

\begin{defn} Let $\FF$ be a foliation on $Y$. Let $\Pi \in
\AA_E(TY)$ and $\Pi : TY = G_\Pi \oplus T\FF$ be the
$\Pi$-splitting. We call a basic vector field $\xi$ tangent to $G_\Pi$
a {\it $\Pi$-basic vector field} or a {\it $G$-basic vector field}.
\end{defn}

In this point of view, the vector field
$$
Y_i: = \frac{\del}{\del y^i} + \sum_{\alpha
=1}^{n-k} R_i^\alpha \frac{\del}{\del q^\alpha}
$$
is the unique $G$-basic vector field that satisfies
$$
Y_j \equiv \frac{\del}{\del y^i} \mod T\FF,
$$
i.e., defines the same transverse vector field as $\frac{\del}{\del y^i}$.
\begin{defn}\label{PiLie}
Let $X$ be any (local) basic vector field of $\FF$ tangent to
$G_\Pi$. We define the {\it $\Pi$-Lie derivative} of $B$
with respect to $X$ by the formula
\begin{equation}\label{eq:PiLie}
L_X^\Pi B = \sum_{i_1 < \cdots < i_\ell} L_X(
B_{i_1i_2\cdots i_\ell}) dy^{i_1}\wedge \cdots \wedge
dy^{i_\ell}
\end{equation}
where $B_{i_1i_2\cdots i_\ell}$ is a local section of $T\FF$
given by the local representation of $B$
$$
B = \sum_{i_1< \cdots < i_\ell} B_{i_1 \cdots
i_\ell} dy^{i_1}\wedge \cdots \wedge dy^{i_\ell}
$$
in any given foliation coordinates. Here  $B_{i_1 \cdots
i_\ell}$ is the (locally defined) leafwise tangent
vector field given by
$$
B_{i_1 \cdots
i_\ell} = B_{i_1 \cdots i_\ell}^\beta \frac{\del}{\del q^\beta}.
$$
\end{defn}
From now on without mentioning further, we will always
assume that $B$ is locally defined, unless otherwise stated.

\begin{defn}\label{dPi}
For any element $B \in \Gamma(\Lambda^\ell(N^*\FF);TF)$, we define
$$
d^\Pi B \in \Gamma(\Lambda^{\ell+1}(N^*\FF);TF)
$$
by the formula
\begin{equation}\label{eq:dPi}
d^\Pi B = \sum_{j = 1}^{2k} dy^j \wedge L^\Pi_{Y_j} B
\end{equation}
where we call the operator $d^\Pi$ the {\it $\Pi$-differential}.
\end{defn}

For given splitting $\Pi$ and a vector field $\xi$, we denote by
$\xi^\Pi$ the projection of $\xi$ to $G=G_\Pi$, i.e.,
$$
\xi^\Pi = \xi - \Pi(\xi).
$$
Then the definition of $d^\Pi$ can be also given  by the same kind of
formula as that of the usual exterior derivative $d$: For given $B
\in \Omega^k(N^*\FF;T\FF)$ and local sections $\eta_1, \cdots,
\eta_{k+1} \in N_x\FF$, we define
\begin{eqnarray}
d^\Pi B(v_1, & \cdots& , v_k, v_{k+1}) \nonumber \\
& = & \sum_i (-1)^{i-1}X_i(B(\eta_1, \cdots, \widehat{\eta_i},
\cdots, \eta_{k+1}))
\nonumber\\
\, & + & \sum_{i < j}(-1)^{i+j -1}B(\pi([X_i,X_j]), \eta_1, \cdots,
\widehat{\eta_i}, \cdots, \widehat{\eta_j}, \cdots, \eta_{k+1}).
\label{eq:inv-dPi}
\end{eqnarray}
Here $X_i$ is a $\Pi$-basic vector field with $\pi(X_i(x)) =
\eta_i(x)$ for each given point $x \in Y$.

It is straightforward
to check that this definition coincides with (\ref{eq:dPi}).

Next we introduce the analog of the ``bracket''
$$
[ \cdot, \cdot]_\Pi  : \Omega^{\ell_1}(N^*\FF; T\FF)
\otimes \Omega^{\ell_2}(N^*\FF;
T\FF) \to \Omega^{\ell_1 + \ell_2}(N^*\FF;T\FF).
$$
\begin{defn} Let $B \in \Omega^{\ell_1}(N^*\FF; T\FF), \, C \in
\Omega^{\ell_2}(N^*\FF; T\FF)$. We define their bracket
$$
[B,C]_\Pi \in \Omega^{\ell_1 + \ell_2}(N\FF;T\FF)
$$
by the formula
\bea
& & [B,C]_\Pi(v_1,\cdots, v_{\ell_1},
v_{\ell_1+1},\cdots,v_{\ell_1
+ \ell_2 })  \nonumber \\
& = & \sum_{\sigma \in S_n}
\frac{\mbox{sign}(\sigma)}{(\ell_1 +\ell_2)!}
 [B(X_{\sigma(1)}, \cdots, X_{\sigma(\ell_1)}),
C(X_{\sigma(\ell_1 +1)}, \cdots, X_{\sigma(\ell_1 + \ell_2)}]\\
& = & \sum_{\tau \in Shuff(n)} \frac{\mbox{sign}(\tau)}{\ell_1!
\ell_2!} [B(X_{\tau(1)}, \cdots, X_{\tau(\ell_1)}),\nonumber \\
&\quad & \hskip2.0in C(X_{\tau(\ell_1 +1)}, \cdots, X_{\tau(\ell_1
+ \ell_2)}]
\eea
for each $x\in Y$ and $v_i \in N_x\FF$, and $X_i$'s are (local)
$\Pi$-basic vector fields such that $\pi(X_i(x)) =
v_i$ as before. Here $S_n$ is the symmetric group with size $n$
and $Shuff(n) \subset S_n$ is the subgroup of all ``shuffles''. $[\cdot, \cdot]$
is the usual Lie bracket of leafwise vector fields.
\end{defn}

For the case $\ell_1 = \ell_2 = 1$, we derive the coordinate
formula
\begin{equation}\label{eq:[B,C]incoord}
[B, C]_\Pi = \Big(B_i^\alpha \frac{\del C^\beta_j} {\del q^\alpha} -
C^\alpha_j\frac{\del B^\beta_i} {\del q^\alpha}\Big) \frac{\del}{\del q^\beta}
\otimes  dy^i \wedge dy^j .
\end{equation}

With these definitions, we have the following ``Bianchi identity''
in our context.
\begin{prop}
Let $\Pi: TY = G \oplus T\FF$ and $d^\Pi$ be the associated
$\Pi$-differential.  Then we have
\begin{eqnarray*}
d^\Pi F_\Pi & = & 0 \label{eq:Bianchi}\\
(d^\Pi)^2 B & = & [F_\Pi, B]_\Pi.
\label{eq:dPi2} \\
\end{eqnarray*}
\end{prop}

Combining the above discussion, the transformation law
(\ref{eq:coord-relation}) in coordinates is translated into the
following invariant form.

\begin{prop}\label{inv-relation}
Let $\Pi, \, \Pi_0$ be two splittings as in Lemma \ref{relation}
and $B_{\Pi_0\Pi} \in \Gamma(N^*\FF\otimes T\FF)$ be the
associated section. Then we have
\begin{equation}\label{eq:inv-relation}
F_\Pi = F_{\Pi_0} + d^{\Pi_0}B_{\Pi_0\Pi} + [B_{\Pi_0\Pi},
B_{\Pi_0\Pi}]_{\Pi_0}.
\end{equation}
\end{prop}

\subsection{Lie algebroid and its $\bar b$-deformed $E$-cohomology}
\label{sliealgebroid}

We start with recalling the definition of Lie algebroid and its
associated $E$-de Rham complex and $E$-cohomology. The leafwise de
Rham complex $\Omega^\bullet(\FF)$ is a special case of the {\it
$E$-de Rham complex} associated to the general Lie algebroid $E$.

We quote the following definitions from \cite{nest-tsygan}.

\begin{defn}\label{liealgebroid} Let $M$ be a smooth manifold. A
{\it Lie algebroid} on $M$ is a triple $(E,\rho, [\, ,\, ])$,
where $E$ is a vector bundle on $M$, $[\, ,\, ]$ is a Lie algebra
structure on the sheaf of sections of $E$, and $\rho$ is a bundle
map, called the {\it anchor map},
$$
\rho: E \to TM
$$
such that the induced map
$$
\Gamma(\rho): \Gamma(M;E) \to \Gamma(TM)
$$
is a Lie algebra homomorphism and, for any sections $\sigma$ and
$\tau$ of $E$ and a smooth function $f$ on $M$, the identity
$$
[\sigma, f\tau] = \rho(\sigma)[f]\cdot \tau + f\cdot [\sigma,
\tau].
$$
\end{defn}

\begin{defn}\label{E-deRham} Let
$(E,\rho,[\, ,\, ])$ be a Lie algebroid on $M$. The {\it $E$-de Rham
complex} $(^E\Omega^\bullet(M), ^Ed)$ is defined by
\begin{eqnarray*}
^E\Omega(\Lambda^\bullet(E^*)) & = &\Gamma(\Lambda^\bullet(E^*))\\
^Ed\omega(\sigma_1, \cdots, \sigma_{k+1})
& = &\sum_i(-1)^i \rho(\sigma_i)\omega(\sigma_1, \cdots,
\widehat{\sigma_i}, \cdots, \sigma_{k+1}) \\
&\qquad &+
\sum_{i<j}(-1)^{i+j-1}\omega([\sigma_i,\sigma_j],\sigma_1,
\cdots,\widehat \sigma_i, \cdots, \widehat \sigma_j, \cdots,
\sigma_{k+1}).
\end{eqnarray*}
The cohomology of this complex will be denoted by $^EH^*(M)$ and
called the $E$-de Rham cohomology of $M$.

Now assume that $\bar b \in ^E\Omega^1(E^*)$ is a  cocycle: $ ^Ed\bar  b = 0$.
Then $^Ed^{\bar b} : = ^Ed +\bar b \wedge$ satisfies: $(^Ed^{\bar b})^2 =0$.
The cohomology of $(^E\Omega^\bullet(M), ^Ed^{\bar b})$ will be denoted by $^EH^*_{\bar b}(M)$ and
called {\it the $\bar b$-deformed $E$-de Rham cohomology}.
\end{defn}

In \cite{oh-park}  the authors have noticed that for a coisotropic submanifold $Y$ in a symplectic manifold $X$ the triple
$$
(E=TY^\omega, \rho = i, [\, ,\, ])
$$
defines the structure of Lie algebroid and  the
$E$-differential 
is   the exterior
derivative $d_\FF$ along the null foliation $\FF$.  Now assume that $(Y, \omega, b)$ is  a coisotropic submanifold in $(X, \omega_X, \alpha)$. 
Then the restriction $\bar b$   of $b$ to $\FF$ is a closed 1-form in the complex $(\Omega (\Lambda ^\bullet E), d_\FF)$ and $^Ed^{\bar b}$
coincides with $d_\FF ^{\bar b}$, which we also denote by $d_\FF ^{ b}$.

The $\bar b$-deformed $E$-de Rahm  differential  is related to  the infinitesimal deformation space of coisotropic
submanifolds in a l.c.s. manifold.  For this, we introduce the space
$$
{\mathcal Coiso}_k = {\mathcal Coiso}_k(X,\omega_X)
$$
the set of coisotropic submanifolds with nullity $n-k$ for $0\leq
k \leq n$ and characterize its infinitesimal deformation space at
$Y \subset E^*$, the zero section of $E^*$.
By the coisotropic neighborhood theorem, the infinitesimal
deformation space, denoted as  $T_{Y}{\mathcal
Coiso}_k(X,\omega_X) = T_{Y}{\mathcal Coiso}_k(U,\omega_U)$ with
some abuse of notion, depends only on $(Y, \omega)$ where $\omega
= i^*\omega_X$, but not on $(X,\omega_X)$. An element in
$T_{Y}{\mathcal Coiso}_k(U,\omega_U)$ is a section of the bundle
$E^*= T^*\FF \to Y$.

\end{document}